\documentclass[12pt,a4paper]{article}
\usepackage{amsmath,amssymb,amstext,amsfonts,latexsym,amsthm}
\usepackage{dsfont,mathrsfs,txfonts}
\usepackage{graphicx,float,hyperref,color}
\usepackage[affil-sl]{authblk}
\usepackage[english,activeacute]{babel}
\usepackage{inputenc}
\usepackage[shortlabels]{enumitem}
\usepackage{orcidlink}

\newtheorem{theorem}{Theorem}
\newtheorem{lemma}{Lemma}[section]

\newtheorem{definition}{Definition}[section]

\newtheorem{remark}{Remark}[section]
\newtheorem{example}{Example}
\newcounter{Th-Alfa}

\newcommand{\CC}{\mathds{C}}

\newcommand{\NN}{\mathds{N}}
\newcommand{\PP}{\mathds{P}}

\newcommand{\RR}{\mathds{R}}
\newcommand{\ZZ}{\mathds{Z}}
\newcommand{\ZZp}{\mathds{Z}_+}
\newcommand{\dsty}{\displaystyle}

\newcommand{\supp}{\mathop{\rm supp}}

\newcommand{\dgr}[1]{\mbox{{\rm deg\/}}(#1)}

\newcommand{\inter}[1]{#1\strut^{\mathrm{o}}\!}
\newcommand{\ch}[1]{\mathbf{Co}\!\left(#1 \right)} 

\def\bbuildrel#1_#2^#3{\mathrel{
 \mathop{\kern 0pt#1}\limits_{#2}^{#3}}}
\def\bbbuildrel#1_#2{\mathrel{
 \mathop{\kern 0pt#1}\limits_{#2}}}

\newcommand{\funD}[4]{\begin{cases} #1 , & \hbox{if } \; #2;\\
										#3, & \hbox{if } \; #4 \end{cases}}
\def\bbuildrel#1_#2^#3{\mathrel{
 \mathop{\kern 0pt#1}\limits_{#2}^{#3}}}
\def\bbbuildrel#1_#2{\mathrel{
 \mathop{\kern 0pt#1}\limits_{#2}}}

\newcommand{\sob}{{\mathsf{s}}}
\newcommand{\Ip}[2]{\langle #1,#2\rangle}
\newcommand{\IpS}[2]{\Ip{#1}{#2}_{\sob}}

\title{\bf Differential properties of Jacobi-Sobolev polynomials and electrostatic interpretation}

\author[1]{H\'{e}ctor Pijeira-Cabrera\,\orcidlink{0000-0002-8465-0646}\thanks{hpijeira@math.uc3m.es}}
\author[1]{Javier Quintero-Roba\,\orcidlink{0000-0003-0536-3102}\thanks{jaquinte@math.uc3m.es}}
\author[2]{Juan Toribio-Milane\,\orcidlink{0000-0002-0782-1827} \thanks{ jtoribio34@uasd.edu.do}\thanks{{The research of J.T.M.  was   partially supported by   Fondo Nacional de Innovaci\'{o}n  y Desarrollo Cient\'{\i}fico y Tecnol\'{o}gico (FONDOCYT),  Dominican Republic, under grant   2020-2021-1D1-137.}}}
\affil[1]{Departamento de Matem\'{a}ticas, Universidad Carlos III de Madrid,  \linebreak  Legan\'{e}s 28911,  Madrid,  Spain.}
\affil[2]{Instituto de Matem\'{a}ticas, Facultad de Ciencias, Universidad Aut\'{o}noma de Santo Domingo, \linebreak Santo Domingo 10105, Dominican Republic.}

\date{}

\begin{document}
\maketitle

\begin{abstract}
We study the sequence of monic polynomials $\{S_n\}_{n\geqslant  0}$, orthogonal with respect to the Jacobi-Sobolev  inner {product}
\;$$
\langle  f,g\rangle_{\mathsf{s}}= \int_{-1}^{1}  f(x) g(x)\, d\mu^{\alpha,\beta}(x)+\sum_{j=1}^{N}\sum_{k=0}^{d_j}\lambda_{j,k} f^{(k)}(c_j)g^{(k)}(c_j),
$$  \;
where $N,d_j \in \ZZ_+$, $\lambda_{j,k}\geqslant  0$, $d\mu^{\alpha,\beta}(x)=(1-x)^{\alpha}(1+x)^{\beta} dx$, $\alpha,\beta>-1$, and $c_j\in\RR\setminus (-1,1)$. A connection formula that relates the Sobolev polynomials $S_n$ with the Jacobi polynomials is provided,  as well as the ladder differential operators for the sequence $\{S_n\}_{n\geqslant  0}$ and a second-order differential  equation with a polynomial coefficient that they satisfied. We give sufficient conditions under which the zeros of a wide class of Jacobi-Sobolev polynomials can be interpreted as the solution of an electrostatic equilibrium problem of $n$ unit charges moving in the presence of a logarithmic potential. Several examples are presented to illustrate this interpretation.

\bigskip



\end{abstract}


\section{Introduction}

It is  well known  that the classical orthogonal polynomials (i.e., Jacobi, Laguerre, and Hermite) satisfy a second-order differential equation with polynomial coefficients, and its zeros are simple. Based on these  facts,  Stieltjes  gave  a very interesting interpretation of the  zeros of the classical orthogonal polynomials as a  solution of an electrostatic equilibrium problem of $n$ movable unit charges in the presence of a logarithmic potential (see \cite[Sec. 3]{VanAss93}.
An excellent introduction to Stieltjes' result on this subject and its consequences can be found in~\cite[Sec. 3]{VanAss93} and \cite[Sec. 2]{ValAss95}. See also the survey \cite{MarMarMar07} and the introduction of  \cite{HMP-PAMS14,OrGa10}.

In order to make this paper self-contained, it is convenient to briefly recall the Jacobi, Laguerre, and Hermite cases. We begin with Jacobi. Let us consider   $n$ unit charges at the points $x_1, x_2, \ldots, x_n$  distributed in   $[-1,1]$ and  add two positive fixed charges of mass $(\alpha + 1)/2$ and $(\beta + 1)/2$ at $1$ and $-1$, respectively. If the charges repel each other according to the logarithmic potential law (i.e., the force is inversely proportional to the relative distance), then the total energy $E(\cdot)$ of this system is obtained by adding the energy of the mutual interaction between the charges. This is \vspace{6pt}
\begin{align}\label{JacobEnergy}
E\left(\omega_1, \omega_2, \ldots, \omega_n\right)=&\sum_{1 \leqslant i<j \leqslant n} \log \frac{1}{\left|\omega_i-\omega_j\right|} + \frac{\alpha + 1}{2} \sum_{j=1}^n \log \frac{1}{\left|1-\omega_j\right|} + \frac{\beta + 1}{2} \sum_{j=1}^n \log \frac{1}{\left|1+\omega_j\right|}.
\end{align}

The minimum of \eqref{JacobEnergy} gives the electrostatic equilibrium. The points $x_1, x_2, \ldots, x_n$ where the minimum is obtained are the places where the charges will settle down. It is obvious that, for the minimum, all the $x_j$ are distinct and different from $\pm 1$.

For a minimum, it is necessary that $\frac{\partial E_t}{\partial \omega_j}=0$ ($1\leqslant k \leqslant n$), from which it follows that the polynomial $P_n(x)=\prod_{j=1}^{n}(x-x_j)$ satisfies the differential equation
\begin{equation}\label{Jacobi-DiffEqn}
\left(1-x^2\right) P_n^{\prime \prime}(x)+\left(\beta-\alpha-(\alpha+\beta+2) x\right) P_n^{\prime}(x)=-n\left(n+\alpha+\beta+1\right) P_n(x),
\end{equation}
which is the differential equation for the monic Jacobi polynomial $P_n(x)=P_n^{\alpha,\beta}(x)$ (see~\cite{Szego75} (Theorems 4.2.2 and 4.21.6)). The proof of the uniqueness of the  minimum, based on the inequality between the arithmetic and geometric means, can be found in~\cite{Szego75} (Section 6.7). In conclusion, the global minimum of \eqref{JacobEnergy} is reached when each of the $n$ charges is located on a zero of the $n$th Jacobi polynomial  $P_n^{\alpha,\beta}(x)$.

For the other two families of classical orthogonal polynomials on the real line (i.e., Laguerre and Hermite),  Stieltjes also gave an electrostatic interpretation.  Since, in this situation, the free charges move in an unbounded set, they can escape to infinity. Stieltjes avoided this situation by constraining the first (Laguerre) or second (Hermite) moment of his zero-counting measures (see  \cite{Szego75} (Theorems 6.7.2 and  6.7.3) and \cite{VanAss93} (Section  3.2)).

The electrostatic interpretation of the zeros of the classical orthogonal polynomials, in addition to Stieltjes, was also studied by B\^{o}cher, Heine, and  Van Vleck, among others. These works were developed between the end of the 19th century and the beginning of the 20th century. After that, the subject remained dormant for almost a century, until it received new impulses from advances in logarithmic potential theory, the extensions of the notion of orthogonality, and the study of new classes of special functions.

{Let $\mu$ be a finite positive Borel measure with finite moments whose support \mbox{$\supp{\mu} \subset \RR$}} contains an infinite set of points. Assume that $\{P_n\}_{n\geqslant  0}$ denotes the monic orthogonal polynomial sequence with respect to the inner product
\begin{align}\label{Standard-IP}
 \Ip{f}{g}_{\mu}= & \int f(x) g(x) d\mu(x).
\end{align}

In general, {an inner product is referred to as``standard''  when the multiplication operator exhibits symmetry with respect to the inner product, i.e., $\langle x f,g\rangle_{\mu}=\langle  f,x g\rangle_{\mu}$. As~\eqref{Standard-IP}  is a  standard inner product, we have that $P_n$  has exactly $n$  simple zeros on  \mbox{$(a,b)=\inter{\ch{\supp{\mu}}}\subset \RR$}}, where $\ch{A}$ denotes the convex hull of a real set $A$ and $\inter{A}$ denotes the interior set of $A$. Furthermore, the sequence $\{P_n\}_{n\geqslant  0}$  satisfies the three-term recurrence relation
$$
   xP_{n}(x)=P_{n+1}(x)+\gamma_{1,n}P_{n}(x)+\gamma _{2,n}P_{n-1}(x); \quad  P_{0} (x)=1, \; P_{-1}(x)=0,
$$
where $\gamma_{2,n}={\|P_{n}\|_{\mu}}^2/{\|P_{n-1}\|_{\mu}}^2$ for $n\geqslant  1$,  $\gamma_{1,n}={\Ip{P_n}{xP_n}_{\mu} }/{\|P_{n}\|_{\mu}^2}$, and  $\|\cdot\|_{\mu}=\sqrt{\Ip{\cdot}{\cdot}_{\mu} }$ denotes the norm induced by \eqref{Standard-IP}.
  See \cite{Chi78,Freud71,Szego75} for these and other properties of $\{P_n\}_{n\geqslant  0}$.

  Let $(a,b)$ be as above, $N,d_j \in \ZZ_+$, $\lambda_{j,k} \geqslant  0$,  for $j=1,\dots, N$, $k=0,1,\dots,d_j$, $\{c_1,\dots,c_N\}\subset \RR \!\setminus\!(a,b)$, where $c_i \neq c_j$ if $i \neq j$ and $I_+=\{(j,k): \lambda_{j,k}>0\}$.  We consider the following   Sobolev-type inner product: \vspace{6pt}
\begin{align}\nonumber
	\IpS{f}{g}&= \Ip{f}{g}_{\mu}+\sum_{j=1}^{N}\sum_{k=0}^{d_j}\lambda_{j,k} f^{(k)}(c_j)g^{(k)}(c_j)\\ \label{GeneralSIP}
	&=\!\int\!f(x) g(x)d\mu(x)+\sum_{(j,k)\in I_+}\!\!\!\!\lambda_{j,k} f^{(k)}(c_j)g^{(k)}(c_j),
\end{align}	
where $f^{(k)}$ denotes the $k$th derivative of the function $f$.  We also assume, without restriction of generality, that  $\{(j,d_j)\}_{j=1}^N\subset I_+$ and $d_1\leqslant d_2\leqslant \cdots \leqslant d_N$.
 Let us denote by  $S_n$ $(n \in \ZZp)$ the lowest degree monic polynomial that satisfies
\begin{equation}\label{Sobolev-Orth}
\IpS{x^k}{S_n}= 0, \quad  \text{for } \; k=0,1,\dots,n-1.
\end{equation}

{Henceforth}, we refer to the sequence  $\left\{S_n\right\}_{n\geqslant 0}$ of monic polynomials as the  system of monic Sobolev-type orthogonal polynomials. It is not difficult to see that for all  $n\geqslant  0$, there exists  a unique polynomial $S_n$ of the degree $n$. Note that the coefficients of $S_n$ are the solution of a homogeneous linear system \eqref{Sobolev-Orth} of  $n+1$ unknowns and $n$ equations. The uniqueness is a consequence of the required minimality on the degree. For more details on this type of nonstandard orthogonality, we refer  the reader to  \cite{MaXu15,And01}.

It is not difficult to see that, in general, \eqref{GeneralSIP} is nonstandard, i.e., $\IpS{xp}{q}\neq\IpS{p}{xq}$.
 The properties of orthogonal polynomials concerning standard inner products are distinct from those of Sobolev-type polynomials. For instance, the roots of Sobolev-type polynomials either can be complex or, if real, might lie beyond the convex hull of the measure $\mu$ support, as demonstrated in the following example:
\begin{example} Let
 \begin{align*}
   \IpS{f}{g}= & \int_{-1}^{1} f(x) g(x)dx+f'(-2)g'(-2)+f'(2)g'(2),
 \end{align*}
 then the corresponding third-degree monic Sobolev-type orthogonal polynomial is $S_3(z)=z^3-\frac{183}{20} z$,  whose zeros are $0$ and $\pm \sqrt{\frac{183}{20}} $. Note that $\pm \sqrt{\frac{183}{20}}  \approx \pm 3 \not \in [-2,2]$.
\end{example}

We will denote by $\PP$ the linear space of all polynomials and by $\dgr{p}$ the degree of $p\in \PP$. Let
\begin{align*}\label{Mod-InnerP}
\widehat{\rho}(x)= &  \prod_{c_j\leqslant a}^{} \left( x-c_j\right)^{d_j+1}\!\!\prod_{c_j\geqslant  b}^{} \left( c_j-x\right)^{d_j+1}  \quad
    \text{ and } \quad  d\mu_{\widehat{\rho}}(x)= \widehat{\rho}(x)d\mu(x).
\end{align*}
{Note that} $\widehat{\rho}(x)>0$ for all  $x \in (a,b)$ and $\dgr{\widehat{\rho}}=d=\sum_{j=1}^N(d_j+1)$.  Additionally, for $n >d$, from~\eqref{Sobolev-Orth}, we have that $\{S_n\}$  satisfies the following quasi-orthogonality relations:
\begin{equation*}
\Ip{S_n}{f}_{\mu_{\widehat{\rho}}}  = \Ip{S_n}{{\widehat{\rho}} f}_{\mu} = \int S_n(x) f(x) {\widehat{\rho}}(x) d\mu(x) = \IpS{S_n}{{\widehat{\rho}} f}=0 ,
\end{equation*}
for $f\in \PP_{n-d-1}$, where $\PP_n$ denotes the linear space of polynomials with real coefficients and degree less than or equal to $n\in \ZZp$. Thus, $S_n$   is a quasi-orthogonal of order  $d$  with respect to the modified measure  $\mu_{\widehat{\rho}}$. Therefore,  $S_n$ has at least $(n-d)$ changes of sign in $(a,b)$.

Taking into account the known results for  measures  of  bounded support (see  \cite{LopMarVan95} (1.10)),  the number of zeros located in the interior of the support of the measure is closely related to $d^*=\#(I_+)$, where the symbol $\#(A)$ denotes  the cardinality of a given set $A$. Note that  $d^*$ is  the number of terms in the discrete part of $\IpS{\cdot}{\cdot}$ ( i.e., $\lambda_{j,k}>0$).

From  Section \ref{LadderOp}  onward,  we will restrict our attention  to the case when in \eqref{GeneralSIP} the measure $d\mu$ is the Jacobi measure $d\mu^{\alpha,\beta}(x)=(1-x)^{\alpha}(1+x)^{\beta} dx$ ($\alpha,\beta>-1$) on  $[-1,1]$. Some of the results we obtain are generalizations of previous work, with  derivatives up to order one. For more details, we refer the reader to \cite{ArAlMa98,DuGa13} and the references therein.

The aim of this paper is   to give an electrostatic interpretation for the distribution of zeros of  a wide class of Jacobi-Sobolev  polynomials, following an approach based on the works \cite{HMP-PAMS14,Ism00-A,Ism00-B} and the original ideas of Stieltjes in \cite{Sti-1885-1,Sti-1885-2}.

In the next section, we obtain a formula that allows us to express the polynomial $S_n$  as a linear combination of $P_{n}$ and $P_{n-1}$, whose coefficients are rational functions.  We refer to this formula as  ``connection formula''. Sections \ref{LadderOp} and \ref{Sec-DiffEqn} deal  with   the ladder (raising and lowering) equations and operators of $\{S_{n}\}_{n\geqslant  0}$. We combine the ladder (raising and lowering) operators to prove that the sequence  of monic polynomials  $\{\widehat{S}_{n}(x)\}_{n\geqslant  0}$ satisfies the second-order linear differential Equation \eqref{DifEq-PolyCoef}, with polynomial coefficients.

In the last section, we give a sufficient condition for  an electrostatic interpretation of the distribution of the zeros of $\{\widehat{S}_{n}(x)\}_{n\geqslant  0}$ as the logarithmic potential interaction of unit positive charges in the presence of an external field. Several examples are given to illustrate whether or not this condition is satisfied.

\section{Connection Formula}\label{ConecFormula}

 {Let $\mu$ be a finite positive Borel measure with finite moments,  whose support \mbox{$\supp{\mu} \subset \RR$}} contains an infinite set of points.  Assume that  $\{P_n\}_{n\geqslant  0}$ denotes the monic orthogonal polynomial sequence with respect to the inner product \eqref{Standard-IP}. We first recall the well-known Christoffel-Darboux formula for $K_{n}(x,y)$, the kernel polynomials associated with  $\{P_n\}_{n\geqslant  0}$.
 \begin{equation}\label{Kernel-n}
K_{n-1}(x,y)=\sum_{k=0}^{n-1}\frac{P_{k}(x)P_{k}(y)}{\left\|
P_{k}\right\|_{\mu} ^{2}}= \funD{\dsty \frac{P_{n}(x)P_{n-1}(y)-P_{n}(y)P_{n-1}(x)}{\|P_{n-1}\|_{\mu}^{2}\, (x-y)}}{x\neq y}{\dsty \frac{P_{n}^{\prime}(x)P_{n-1}(x)-P_{n}(x)P_{n-1}^{\prime}(x)}{\|P_{n-1}\|_{\mu}^{2}}}{x=y.}
\end{equation}

  We denote by  $\dsty K_{n}^{(j,k)}\left( x,y\right) =\frac{\partial ^{j+k}K_{n}\left( x,y\right) }{\partial x ^{j}\partial y^{k}}$ the partial derivatives of  the kernel \eqref{Kernel-n}. Then, from the Christoffel-Darboux formula \eqref{Kernel-n} and the Leibniz rule, it is not difficult to verify~\mbox{that}
\begin{align}\nonumber
K_{n-1}^{(0,k)}(x,y)=& \sum_{i=0}^{n-1}\frac{P_{i}(x)P_i^{(k)}(y)}{\left\| P_{i }\right\| _{\mu }^{2}} \\ =& \frac{k!\left(
Q_{k}(x,y;P_{n-1})P_{n}(x)-Q_{k}(x,y;P_{n})P_{n-1}(x)\right)}{\left\| P_{n-1}\right\|
_{\mu }^{2}\;(x-y)^{k+1}}
, \label{CDformuladj}
\end{align}
where $Q_{k}(x,y;f)=\sum_{\nu=0}^{k} \frac{f^{(\nu)}(y)}{\nu !}(x- y)^\nu$   is  the Taylor polynomial of the degree $k$ of $f$ centered at $y$. Observe that \eqref{CDformuladj} becomes the usual Christoffel-Darboux formula \eqref{Kernel-n} if  $k=0$.

From \eqref{GeneralSIP}, if $i<n$
\begin{equation}\label{FourierCoeff}
  \Ip{S_n}{P_i}_{\mu }=  \IpS{S_n}{P_i}-\sum_{(j,k)\in I_+}\lambda_{j,k} S_n^{(k)}(c_j)P_i^{(k)}(c_j)=-\sum_{(j,k)\in I_+}\lambda_{j,k} S_n^{(k)}(c_j)P_i^{(k)}(c_j).
\end{equation}

Therefore, from the Fourier expansion of $S_{n}$ in terms of the basis $\left\{ P_{n}\right\} _{n\geqslant 0}$ and using \eqref{FourierCoeff}, we~obtain \vspace{-6pt}
\begin{align}
S_{n}(x)&=P_n(x)+\sum_{i=0}^{n-1}\Ip{S_n}{P_i}_{\mu }\frac{P_{i}(x)}{\left\| P_{i }\right\| _{\mu }^{2}}= P_n(x)-\sum_{(j,k)\in I_+}\lambda_{j,k} S_n^{(k)}(c_j)\sum_{i=0}^{n-1}\frac{P_{i}(x)P_i^{(k)}(c_j)}{\left\| P_{i }\right\| _{\mu }^{2}} \nonumber
\\
&=P_{n}(x)-\sum_{(j,k)\in I_+}\lambda_{j,k} S_{n}^{(k)}(c_{j})K_{n-1}^{(0,k)}(x,c_j).\label{FConexFINAL}
\end{align}

Now, replacing \eqref{CDformuladj} in \eqref{FConexFINAL}, we have  the  connection formula
\begin{align}\label{FConexPPAL}
S_{n}(x)& = F_{1,n}(x)P_{n}(x) + G_{1,n}(x)P_{n-1}(x), \\ \nonumber
\text{where} \quad  F_{1,n}(x) & = 1-\sum_{(j,k)\in I_+} \frac{\lambda _{j,k}k!\,S_{n}^{(k)}(c_{j})}{\left\| P_{n-1}\right\|_{\mu}^{2}}\frac{Q_{k}(x,c_{j};P_{n-1})}{(x-c_{j})^{k+1}} \\ \nonumber
\text{and} \quad  G_{1,n}(x) =&\sum_{(j,k)\in I_+} \frac{\lambda _{j,k}k!\,S_{n}^{(k)}(c_{j})}{\left\| P_{n-1}\right\|_{\mu}^{2}}\frac{Q_{k}(x,c_{j};P_{n})}{(x-c_{j})^{k+1}}.
\end{align}

Deriving Equation~\eqref{FConexFINAL} $\ell$-times and evaluating then at $x=c_{i}$ for each ordered pair $(i,\ell)\in I_+$, we obtain the following system of $d^*=\#(I_+)$ linear equations and $d^*$ unknowns $S^{(k)}_n(c_{j})$.
\begin{align}\label{eqsystem1}
P_{n}^{(\ell)}(c_{i})=\left(1+\lambda_{i,\ell}K_{n-1}^{(\ell,\ell)}(c_i,c_i)\right)S^{(\ell)}_n(c_{i})\ +\!\!\!\sum_{\substack{(j,k)\in I_+\\ (j,k)\neq (i,\ell)}}\!\!\!\!\!\lambda_{j,k} K_{n-1}^{(\ell,k)}(c_{i},c_j)S^{(k)}_n(c_{j}).
\end{align}

The remainder of this section is devoted to proving that system \eqref{eqsystem1} has a unique solution. The following  lemma  is essential to achieve this goal.

\begin{lemma}\label{LemmaFullRank}
Let $I\subset \mathbb{R}\times \mathbb{Z}_+$ be a (finite) set of $d^*$ pairs. Denote $\{c_j\}_{j=1}^N=\pi_1(I)$ where $\pi_1$ is the projection function over the first coordinate, i.e., $\pi_1(x,y)=x$, $d_j=\max\{\nu_i: (c_j,\nu_i)\in I\}$ and $d=\sum_{j=1}^N (d_j+1)$. Let $P_k$ be an arbitrary polynomial of the degree $k$ for $0\leqslant k\leqslant n-1$. Then,  for all $n\geqslant  d$, the $ d^*\!\times \!n$ matrix
\begin{align*}
\mathcal{A}^*=\left(P^{(\nu)}_{k-1}(c)\right)_{(c,\nu)\in I,k=1,2,\dots,n}
\end{align*}
has a full rank $d^*$.
\end{lemma}
\begin{proof}
First, note that, using elementary column transformations, we can reduce the proof to the case when $P_k(x)=x^k$, for $k=0,1,\dots,n-1$. On the other hand, $d^*_j=\#(\{\nu_i: (c_j,\nu_i)\in I\})\leqslant d_j+1$ for $j=1,2,\dots,N$, so $d^*=\sum_{j=1}^Nd_j^*\leqslant d\leqslant n$, and it is sufficient to prove the case $n=d$. Consider the $m\!\times\! n$ matrix
\begin{align*}
\mathcal{A}_{m}(x)=\left(\begin{array}{rrrrrr}
1 & x & x^2 & x^3  &\cdots & x^{n-1}\\
0 & 1 & 2x & 3x^2  &\cdots & (n-1)x^{n-2} \\
0 & 0 & 2 & 6 x  &\cdots & (n-1)(n-2)x^{n-3} \\
\vdots & \vdots &\vdots & \ddots &\vdots & \vdots\\
0 & 0 &0 & 0 &\cdots & (n-1)\cdots(n-m+1)x^{n-m}\\
\end{array}\right),
\end{align*}
where $m\leqslant n$.  Without loss of generality, we can rearrange the rows of $\mathcal{A}^*$ such that
\begin{align*}
\mathcal{A}^*=
\begin{bmatrix}
    \begin{array}{c}
        \mathcal{A}_1^* \\
        \hline
        \mathcal{A}_2^* \\
        \hline
        \vdots\\
        \hline
        \mathcal{A}_N^* \\
    \end{array}
\end{bmatrix}, \quad  \text{where} \quad   \mathcal{A}^*_j=\left(P^{(\nu)}_{k-1}(c_j)\right)_{(c_j,\nu)\in I,k=1,2,\dots,n}.
\end{align*}

 Note that $\mathcal{A}^*_j$ is obtained by taking some rows from $\mathcal{A}_{d_j+1}(c_j)$, the rows $\nu$, such that $(c_j,\nu-1)\in I$.
 Consider the matrix \vspace{6pt}
\begin{align*}
\mathcal{A}=
 \begin{bmatrix}
    \begin{array}{c}
        \mathcal{A}_{d_1+1}(c_1) \\
        \hline
        \mathcal{A}_{d_2+1}(c_2) \\
        \hline
        \vdots\\
        \hline
        \mathcal{A}_{d_N+1}(c_N) \\
    \end{array}
\end{bmatrix}.
\end{align*}
 From \cite{Kratt99} (Theorem 20), we compute  $det(\mathcal{A})$ as
\begin{align*}
\det(\mathcal{A})=\det(\mathcal{A}^\intercal)=\prod_{j=1}^N \prod_{i=1}^{d_j}i!\prod_{1\leqslant j_1<j_2 \leqslant N}(c_{j_1}-c_{j_2})^{(d_{j_1}+1)(d_{j_2}+1)}\neq 0.
\end{align*}
Then the $n$ row vectors of $\mathcal{A}$ are linearly independent, and consequently, the $d^*$ rows of $\mathcal{A}^*$ are also linearly independent.
\end{proof}

Now we can rewrite  \eqref{eqsystem1} in the matrix form
\begin{align}\label{solve1}
\mathcal{P}_{n}(\mathcal{C})=(\mathcal{I}_{d^{*}}+\mathcal{K}_{n-1}(\mathcal{C},\mathcal{C})\mathcal{L})\mathcal{S}_{n}(\mathcal{C}), \quad \text{where}
\end{align}

\begin{description}
  \item[$\mathcal{I}_{d^{*}}$] is the identity matrix of the order $d^*$.
    \item[$\mathcal{L}$ ]  is the $d^*\!\!\times\! d^*$-diagonal matrix with the diagonal entries $\lambda_{j,k}$, $(j,k)\in I_+ $.
  \item[$\mathcal{C}$] is the column vector $\dsty \mathcal{C}=(\underbrace{c_{1},\dots,c_1}_{d^*_1 \text{-times}},\underbrace{c_2,\dots,c_2}_{d^*_2 \text{-times}},\dots,
\underbrace{c_N,\dots,c_N}_{d^*_N \text{-times}})^\intercal$.
  \item[$\mathcal{P}_{n}(\mathcal{C})$] \!and $\mathcal{S}_{n}(\mathcal{C})$  are column vectors with the entries $P^{(k)}_n(c_j)$, and $S^{(k)}_n(c_j)$, $(j,k)\in I_+$ respectively.
  \item[$\mathcal{K}_{n-1}(\mathcal{C},\mathcal{C})$]  is a $d^*\times d^*$ matrix whose entry associated to the $(i,\ell)$th row and the $(j,k)$th column, $(i,\ell),(j,k)\in I_+$, is \;
$\dsty
K_{n-1}^{(\ell,k)}(c_{i},c_{j})=\sum_{\nu=0}^{n-1}\frac{P^{(\ell)}_{\nu}(c_i)P_{\nu}^{(k)}(c_j)}{\left\|
P_{\nu}\right\|_{\mu} ^{2}}.$
\end{description}

Clearly,  we can write $\dsty \mathcal{K}_{n-1}(\mathcal{C},\mathcal{C})=\mathcal{F}\mathcal{F}^\intercal,$ where $\mathcal{F}=\left(\frac{P^{(k)}_{\nu-1}(c_j)}{\|P_{\nu-1}\|_\mu}\right)_{(j,k)\in I_+, \nu=1,\dots,n,}$ is a matrix  of the order $d^*\!\!\times \! n$ and  full rank  for all $n\geqslant  d$, according to   Lemma \ref{LemmaFullRank}.

Then the matrix $\mathcal{K}_{n-1}(\mathcal{C},\mathcal{C})$ is a $d^*\!\!\times\! d^*$ positive definite matrix for all $n\geqslant  d$; see~\cite{Horn90} (Theorem 7.2.7(c)). Since $\mathcal{L}$ is a diagonal matrix with positives entries, it follows that $\mathcal{L}^{-1}+\mathcal{K}_{n-1}(\mathcal{C},\mathcal{C})$ is also a positive definite matrix, and consequently, $\mathcal{I}_{d^{*}}+\mathcal{K}_{n-1}(\mathcal{C},\mathcal{C})\mathcal{L}=\left(\mathcal{L}^{-1}+\mathcal{K}_{n-1}(\mathcal{C},\mathcal{C})\right)\!\mathcal{L}$ is nonsingular. Then the linear system  \eqref{solve1} has the unique solution
\begin{align}\label{Solut1}
\mathcal{S}_{n}(\mathcal{C})=(\mathcal{I}_{d^{*}}+\mathcal{K}_{n-1}(\mathcal{C},\mathcal{C})\mathcal{L})^{-1}\mathcal{P}_{n}(\mathcal{C}).
\end{align}

Using this notation, we can rewrite  \eqref{FConexFINAL} in the compact form
\begin{equation}\label{formula1}
S_{n}(x)=P_{n}(x)-\mathcal{K}_{n-1}(x,\mathcal{C})\,\mathcal{L} \,\mathcal{S}_{n}(\mathcal{C}),
\end{equation}
where $\mathcal{K}_{n-1}(x,\mathcal{C})$ is a row vector with the entries $K_{n-1}^{(0,k)}(x,c_{j})$, for $(j,k)\in I_+$.
Now, replacing~\eqref{Solut1} into~\eqref{formula1}, we obtain the matrix version of the connection formula \eqref{FConexPPAL}
\begin{equation*}
S_{n}(x)=P_{n}(x)-\mathcal{K}_{n-1}(x,\mathcal{C})\mathcal{L}(\mathcal{I}_{d^{*}}+\mathcal{K}_{n-1}(\mathcal{C},\mathcal{C})\mathcal{L})^{-1}\mathcal{P}_{n}(\mathcal{C}).
\end{equation*}


\section{Ladder Equations  for Jacobi-Sobolev Polynomials}\label{LadderOp}

Henceforth,  we will restrict our attention to  the Jacobi-Sobolev case. Therefore, we consider in the inner product  \eqref{GeneralSIP} the measure $d\mu(x)=d\mu^{\alpha,\beta}(x)=(1-x)^{\alpha}(1+x)^{\beta} dx$, where $\alpha,\beta>-1$ and whose support is $[-1,1]$. To simplify  the notation, we will continue to write  $S_n$ instead of $S^{\alpha,\beta}_n$ to denote the corresponding $n$th Jacobi-Sobolev monic polynomial.   In the following, we omit the parameters $\alpha$ and $\beta$ when no confusion arises.

From \cite{Szego75} ((4.1.1), (4.3.2), (4.3.3), (4.5.1), and (4.21.6)), for the monic Jacobi polynomials, we have
 \begin{align}\nonumber
P_n^{\alpha, \beta}(x)=& \binom{2 n+\alpha+\beta}{n}^{-1} \sum_{\nu=0}^n\binom{n+\alpha}{n-\nu}\binom{n+\beta}{\nu}\left(x-1\right)^\nu\left(x+1\right)^{n-\nu}.\\ \nonumber
h^{\alpha,\beta}_n=  \left\|P^{\alpha,\beta}_{n}\right\|_{\mu^{\alpha,\beta}}^2 = & 2^{2n+\alpha+\beta+1}\frac{\Gamma(n+1)\Gamma(n+\alpha+1)\Gamma(n+\beta+1)\Gamma(n+\alpha+\beta+1)}
{\Gamma(2n+\alpha+\beta+2)\Gamma(2n+\alpha+\beta+1)}.\\ \nonumber
P_n^{ \alpha, \beta }(1)=& \frac{2^n \Gamma(n+\alpha+1) \Gamma(n+\alpha+\beta+1)}{\Gamma(\alpha+1) \Gamma(2 n+\alpha+\beta+1)} .\\
 xP^{\alpha,\beta}_{n}(x)=&P^{\alpha,\beta}_{n+1}(x)+\gamma_{1,n}P^{\alpha,\beta}_{n}(x)+\gamma _{2,n}P^{\alpha,\beta}_{n-1}(x); \;  P^{\alpha,\beta}_{0} (x)=1, \; P^{\alpha,\beta}_{-1}(x)=0, \label{Jacob3TRR}
\end{align}
where
\begin{equation}\label{3TRR-CoefJacob}
\begin{aligned}
 \gamma_{1,n}=& \gamma^{\alpha,\beta}_{1,n}= \frac{\beta^2-\alpha^2}{(2n+\alpha+\beta)(2n+\alpha+\beta+2)},\\
\gamma_{2,n}= &\gamma^{\alpha,\beta}_{2,n}= \frac{4n (n+\alpha)(n+\beta) (n+\alpha+\beta)}{(2 n+\alpha+\beta)^2((2 n+\alpha+\beta)^2-1)}.
\end{aligned}
\end{equation}

Let $\mathfrak{I}$ be  the identity operator. We define the two ladder Jacobi differential operators on $\PP$ as
\begin{align*}
\widehat{\mathfrak{L}}_{n}^{\downarrow} & :=-\frac{\widehat{a}_{n}(x)}{\widehat{b}_{n}}\; \mathfrak{I}+\frac{1-x^2}{\widehat{b}_{n}} \;\frac{d}{dx}\quad  \text{(lowering Jacobi differential operator)},\\
 \widehat{\mathfrak{L}}_{n}^{\uparrow} &:= -\frac{\widehat{c}_{n}(x)}{\widehat{d}_{n}}\; \mathfrak{I}+\frac{1-x^2}{\widehat{d}_{n}} \; \frac{d}{dx} \quad \text{(raising Jacobi differential operator)} .
\end{align*}
where
\begin{equation}\label{coeffSzego01}
\begin{aligned}
\widehat{a}_{n}(x)=&-\frac{n ((2n+\alpha+\beta)x+\beta-\alpha)}{2n+\alpha+\beta}, \quad
\widehat{b}_{n}= \frac{4n (n+\alpha)(n+\beta)(n+\alpha+\beta)}{(2 n+\alpha+\beta)^2(2 n+\alpha+\beta-1)},\\
\widehat{c}_{n}(x)=&\frac{(n+\alpha+\beta)((2 n+\alpha+\beta) x+\alpha-\beta)}{2n+\alpha+\beta} \quad \text{and} \quad
\widehat{d}_{n} = -(2n+\alpha+\beta-1).
\end{aligned}
\end{equation}

From \cite{Szego75} (4.5.7 and 4.21.6), if  $n\geqslant  1$, the sequence $\left\{P^{\alpha,\beta}_{n}\right\}_{n\geqslant  0}$  satisfies the  relations
\begin{equation}\label{StrcRelPn}
\begin{aligned}
\widehat{\mathfrak{L}}_{n}^{\downarrow}  \left[P^{\alpha,\beta}_{n}(x)\right] & = -\frac{\widehat{a}_{n}(x)}{\widehat{b}_{n}}P^{\alpha,\beta}_{n}(x)+\frac{1-x^2}{\widehat{b}_{n}} \left(P^{\alpha,\beta}_{n}(x)\right)^{\prime} = P^{\alpha,\beta}_{n-1}(x),\\
\widehat{\mathfrak{L}}_{n}^{\uparrow}  \left[P^{\alpha,\beta}_{n-1}(x)\right] & = -\frac{\widehat{c}_{n}(x)}{\widehat{d}_{n}}P^{\alpha,\beta}_{n-1}(x)+\frac{1-x^2}{\widehat{d}_{n}} \left(P^{\alpha,\beta}_{n-1}(x)\right)^{\prime}= P^{\alpha,\beta}_{n}(x).
\end{aligned}
\end{equation}

In this case, the connection Formula \eqref{FConexPPAL} becomes
\begin{align}\label{FConexJ-S}
  S_{n}(x)=& A_{1,n}(x) \; P^{\alpha,\beta}_{n}(x)+ B_{1,n}(x)\; P^{\alpha,\beta}_{n-1}(x),\\ \nonumber
  \text{where} \quad  A_{1,n}(x) =&  A^{\alpha,\beta}_{1,n}(x)=   1-\sum_{(j,k)\in I_+} \frac{\lambda _{j,k}k!\,S_{n}^{(k)}(c_{j})}{h^{\alpha,\beta}_{n-1}}\frac{Q_{k}(x,c_{j};P^{\alpha,\beta}_{n-1})}{(x-c_{j})^{k+1}}\\
\nonumber
  \text{and} \quad  B_{1,n}(x) =&  B^{\alpha,\beta}_{1,n}(x)=   \sum_{(j,k)\in I_+} \frac{\lambda _{j,k}k!\,S_{n}^{(k)}(c_{j})}{h^{\alpha,\beta}_{n-1}}\frac{Q_{k}(x,c_{j};P^{\alpha,\beta}_{n})}{(x-c_{j})^{k+1}}.
\end{align}

Let  $\dsty \rho(x)=\prod_{j=1}^{N} \left( x-c_j\right)^{d_j+1}$   and define the $\left(d-k-1\right) $th degree polynomial%
\begin{equation}\label{rhoMk}
\rho _{j,k}(x):=\frac{\rho(x)}{(x-c_j)^{k+1}}=(x-c_{j})^{d_j-k}\prod_{\substack{i=1 \\ i\neq j}}^N(x-c_{i})^{d_i+1},
\end{equation}
for every $(j,k)\in I_+$.  The following four lemmas are essential for defining ladder operators (lowering and raising operators).



\begin{lemma}
\label{lemma41} For the sequences of polynomials $\{S_{n}\}_{n\geqslant  0}$ and $\{P^{\alpha,\beta}_{n}\}_{n\geqslant  0}$, we obtain%
\begin{eqnarray}
\rho(x)S_{n}(x) &=& A_{2,n}(x) \;  P^{\alpha,\beta}_{n}(x)+B_{2,n}(x) \;P^{\alpha,\beta}_{n-1}(x),
\label{LemConnForm1} \\
\left(1-x^2\right)\left( \rho (x)S_{n}(x)\right) ^{\prime }
&=&A_{3,n}(x)  P^{\alpha,\beta}_{n}(x)+B_{3,n}(x)  P^{\alpha,\beta}_{n-1}(x),  \label{LemConnForm1-Dx}
\end{eqnarray}%
where%
\begin{align*}
A_{2,n}(x)  =&\rho(x)A_{1,n}(x)   =\rho(x)-\sum_{(j,k)\in I_+}\left(\frac{k! \lambda
_{j,k}S_{n}^{(k)}(c_{j})}{h^{\alpha,\beta}_{n-1}}\,Q_{k}\left(x,c_{j};P^{\alpha,\beta}_{n-1}\right)\right) \rho _{j,k}(x), \\
B_{2,n}(x)  =&\rho(x)B_{1,n}(x)   = \sum_{(j,k)\in I_+}\left(\frac{k! \lambda _{j,k}S_{n}^{(k)}(c_{j})}{h^{\alpha,\beta}_{n-1}}%
\,Q_{k}\left(x,c_{j};P^{\alpha,\beta}_{n}\right)\right) \rho _{j,k}(x), \\
A_{3,n}(x) = &  A_{2,n}'(x) \left(1-x^2\right) +\widehat{a}_n(x)A_{2,n}(x)+\widehat{d}_{n}B_{2,n}(x), \\
B_{3,n}(x) = & B_{2,n}'(x) \left(1-x^2\right) +\widehat{b}_n A_{2,n}(x)+\widehat{c}_{n}(x)B_{2,n}(x),
\end{align*}
where $A_{2,n}$, $B_{2,n}$, $A_{3,n}$, and $B_{3,n}$ are polynomials of degree at most $d$, $d-1$, $d+1$ and $d$, respectively,  and the coefficients  $\widehat{a}_n(x)$, $\widehat{b}_n$, $\widehat{c}_n(x)$, and $\widehat{d}_n$   are given by \eqref{coeffSzego01}.
\end{lemma}



\begin{proof}
From  \eqref{FConexJ-S} and \eqref{rhoMk}, Equation \eqref{LemConnForm1} is immediate. To prove \eqref{LemConnForm1-Dx}, we can take derivatives with respect to $x$ in both hand
sides of \eqref{LemConnForm1} and then multiply by $1-x^2$%
\begin{align*}
\begin{split}
 \left(1-x^2\right) \left( \rho (x)S_{n}(x)\right) ^{\prime }=& \left(1-x^2\right) A_{2,n}' P_{n}(x)+A_{2,n} \left(1-x^2\right)  \left(P_{n}^{\alpha,\beta}(x)\right)'\\
&+ \left(1-x^2\right) B_{2,n}' P_{n-1}^{\alpha,\beta}(x)+B_{2,n}  \left(1-x^2\right)  \left(P_{n-1}^{\alpha,\beta}(x)\right)'.
\end{split}
\end{align*}%

 Using \eqref{StrcRelPn} in the above expression, we obtain
\begin{align*}
 \left(1-x^2\right) \left(\rho(x)S_{n}(x)\right)^{\prime }=&\left[ A_{2n}'(x) \left(1-x^2\right) +\widehat{a}_n(x)A_{2,n}(x)+B_{2,n}(x)\widehat{d}_{n}\right]P^{\alpha,\beta}_{n}(x)
\\
&+\left[ B_{2,n}'(x) \left(1-x^2\right) +\widehat{b}_n A_{2,n}(x)+B_{2,n}(x)\widehat{c}_{n}(x)\right] P_{n-1}^{\alpha,\beta}(x),
\end{align*}%
which is \eqref{LemConnForm1-Dx}.
\end{proof}



\begin{lemma}
 The sequences of the monic polynomials $\left\{S_{n}\right\}_{n\geqslant  0}$ and
$\left\{P_{n}^{\alpha,\beta}\right\}_{n\geqslant  0}$ are also related by the equations
\begin{align}
\rho (x)S_{n-1}(x) &=C_{2,n}(x) P_{n}^{\alpha,\beta}(x)+D_{2,n}(x) P_{n-1}^{\alpha,\beta}(x),
\label{LemConnForm2} \\
 \left(1-x^2\right) \left(\rho(x)S_{n-1}(x)\right) ^{\prime }
&=C_{3,n}(x)P_{n}^{\alpha,\beta}(x)+D_{3,n}(x) P_{n-1}^{\alpha,\beta}(x),  \label{LemConnForm2-Dx}
\end{align}%
where%
$$
\begin{array}{lllll}
\dsty C_{2,n}(x)  & =  \dsty -\frac{B_{2,n-1}(x)}{\gamma_{2,n-1}},&\quad
\dsty D_{2,n}(x)  & = \dsty A_{2,n-1}(x)+B_{2,n-1}(x)\left(\frac{x-\gamma_{1,n-1}}{\gamma_{2,n-1}}\right),\\
\dsty C_{3,n}(x)  & = \dsty -\frac{B_{3,n-1}(x)}{\gamma_{2,n-1}},&\quad
\dsty D_{3,n}(x)  & = \dsty A_{3,n-1}(x)+B_{3,n-1}(x)\left(\frac{x-\gamma_{1,n-1}}{\gamma_{2,n-1}}\right),\\
\end{array}
$$
where $C_{2,n}(x)$, $D_{2,n}(x)$, $C_{3,n}(x)$, and $D_{3,n}(x)$ are polynomials of degree at most $d-1$, $d$, $d$ and $d+1$, respectively.
\end{lemma}

\begin{proof}
The proof of \eqref{LemConnForm2} and \eqref{LemConnForm2-Dx} is a straightforward consequence of Lemma \ref{lemma41} and the three-term recurrence relation \eqref{Jacob3TRR}, whose coefficients are given in \eqref{3TRR-CoefJacob}.
\end{proof}



\begin{lemma}
 The monic orthogonal Jacobi polynomials $\left\{P_{n}^{\alpha,\beta}\right\}_{n\geqslant  0}$%
 can be expressed in terms of the monic Sobolev-type polynomials $%
\left\{S_{n}\right\}_{n\geqslant  0}$ in the following way:%
\begin{align}
P_{n}^{\alpha,\beta}(x) &=\frac{\rho (x)}{\Delta_n(x)}\left(D_{2,n}(x)S_{n}(x)-B_{2,n}(x)S_{n-1}(x)\right) ,  \label{Ln-desp-01} \\
P_{n-1}^{\alpha,\beta}(x) &=\frac{\rho (x)}{\Delta_n(x)}\left(A_{2,n}(x)S_{n-1}(x)-C_{2,n}(x)S_{n}(x)\right) .  \label{Ln-desp-02}
\end{align}%
where%
\begin{equation}\label{Main-Determinant}
\Delta_n(x)=\det\begin{pmatrix}
A_{2,n}(x) & B_{2,n}(x) \\
C_{2,n}(x) & D_{2,n}(x)%
\end{pmatrix}%
=A_{2,n}(x)D_{2,n}(x)-C_{2,n}(x)B_{2,n}(x)
\end{equation}%
is a polynomial of the degree  $2d$.
\end{lemma}

\begin{proof}
Note that \eqref{LemConnForm1} and \eqref{LemConnForm2} form  a system of two linear equations with the two unknowns $P_{n}^{\alpha,\beta}(x)$ and $P_{n-1}^{\alpha,\beta}(x)$. Therefore, from Cramer's rule, we obtain \eqref{Ln-desp-01} and \eqref{Ln-desp-02}.

As $\dgr{C_{2,n}\,B_{2,n}}\leqslant 2d-2$ and $\dsty \lim_{x\to \infty} \frac{A_{2,n}(x)}{x^{2d}}=1$, we obtain
\begin{align}\label{Leading-Delta}
 \lim_{x\to \infty} \frac{\Delta_n(x)}{x^{2d}}=& \lim_{x\to \infty} \frac{ D_{2,n(x)}}{x^{d}} =\left\{
                                                                                                 \begin{array}{ll}
                                                                                                   1, & \text{if }\;  \dgr{B_{2,n-1}}<d-1,\\ \dsty
                                                                                                  1+\frac{\Lambda_{n-1}}{\gamma_{2,n-1} \,h^{\alpha,\beta}_{n-2} }, & \text{if }\;  \dgr{B_{2,n-1}}=d-1,
                                                                                                 \end{array}
                                                                                               \right.
\end{align}
where $\dsty\Lambda_{n-1}=\; \sum_{(i,j)\in I_+}\lambda_{k,j}\left(S_{n-1}(c_j)\right)^{(k)}\left(P_{n-1}^{\alpha.\beta}(c_j)\right)^{(k)}= \left(\mathcal S_{n-1}(\mathcal C)\right)^T \mathcal L \mathcal P_{n-1}(\mathcal C)$. From \eqref{solve1},
\begin{align*}
\Lambda_{n-1}& =\left(\mathcal{S}_{n-1}(\mathcal C)\right)^T \mathcal L \left(\mathcal{I}_{d^*}+\mathcal{K}_{n-2}(\mathcal C,\mathcal C)\mathcal L\right)\mathcal{S}_{n-1}(\mathcal C).
\end{align*}
Since the matrix $\mathcal L \left(\mathcal{I}_{d^*}+\mathcal{K}_{n-2}(\mathcal C,\mathcal C)\mathcal L\right)$ is positive definite, we conclude that
\begin{equation}\label{LambdaNpositive}
\Lambda_{n-1}>0, \text{ for all } n\in \NN;
\end{equation}
i.e.,  $\Delta_n(x)$  is a polynomial of the degree  $2d$.
\end{proof}

\begin{remark}  Obviously, from \eqref{Ln-desp-01} (or \eqref{Ln-desp-02}), we have that $\Delta_n(x)= \rho (x) \delta_n(x)$, where $\delta_n$ is a polynomial of the degree  $d$. Hence, from \eqref{Main-Determinant},
\begin{align*}
 \delta_n(x) & =A_{1,n}(x)D_{2,n}(x)-B_{1,n}(x)C_{2,n}(x).
\end{align*}
\end{remark}

\begin{theorem}\label{lemma44} Under the above assumptions,  we have the following ladder equations:
\begin{align}
A_{4,n}(x)S_{n}(x)+B_{4,n}(x)S_{n}^{\prime}(x) & =S_{n-1}(x),
\label{laddereq1} \\
C_{4,n}(x)S_{n-1}(x)+D_{4,n}(x)S_{n-1}^{\prime}(x) & = S_{n}(x),
\label{laddereq2}
\end{align}%
where
\begin{align*}
 A_{4,n}(x)  & = \frac{q_{2,n}(x) }{q_{1,n}(x)},  \;    B_{4,n}(x)  = \frac{q_{0,n}(x)}{q_{1,n}(x)}, \;
C_{4,n}(x)   = \frac{q_{3,n}(x)}{q_{4,n}(x)},\;   D_{4,n}(x)   = \frac{q_{0,n}(x)}{q_{4,n}(x)}.\\
q_{0,n}(x)&= \left(1-x^2\right)\Delta_n(x),\quad \dgr{q_{0,n}}=2d+2.\\
q_{1,n}(x)&= B_{3,n}(x)A_{2,n}(x)-A_{3,n}(x)B_{2,n}(x),\quad \dgr{q_{1,n}}= 2d.\\
q_{2,n}(x) &=(1-x^2)\rho^{\prime}(x)\delta_n(x)+B_{3,n}(x)C_{2,n}(x)- A_{3,n}(x)D_{2,n}(x),\quad \dgr{q_{2,n}}= 2d+1.\\
q_{3,n}(x) & = (1-x^2)\rho^{\prime}(x)\delta_n(x) +C_{3,n}(x)B_{2,n}(x)-D_{3,n}(x)A_{2,n}(x),\quad \dgr{q_{3,n}}= 2d+1.\\
q_{4,n}(x)& = C_{3,n}(x)D_{2,n}(x)-D_{3,n}(x)C_{2,n}(x),  \quad \dgr{q_{4,n}}= 2d.
\end{align*}
\end{theorem}
\begin{proof}
Replacing \eqref{Ln-desp-01} and \eqref{Ln-desp-02} in \eqref{LemConnForm1-Dx}
and \eqref{LemConnForm2-Dx}, the two ladder Equations \eqref{laddereq1} and \eqref{laddereq2}  follow.

\begin{enumerate}
  \item
$$
\lim_{x\to \infty} \frac{q_{1,n}(x)}{x^{2d}}   = \left\{  \begin{array}{ll}      \widehat{b}_{n}, & \text{if }\;  \dgr{B_{2,n}}<d-1,\\  \dsty        \widehat{b}_n+(2n+\alpha+\beta+1)\frac{\Lambda_n}{h_{n-1}^{\alpha,\beta}}, & \text{if }\;  \dgr{B_{2,n}}=d-1,  \end{array}\right.
$$
where, according to \eqref{LambdaNpositive}, $ \dsty  \Lambda _{n} >0$,  i.e.,  $\dgr{q_{1,n}}= 2d.$

\item From \eqref{Leading-Delta}, $\dsty \lim_{x\to \infty} \frac{\delta_n(x)}{x^d}=\lim_{x\to \infty} \frac{D_{2,n}(x)}{x^d}=\kappa_2>0$.
\begin{align*}
\lim_{x\to \infty} \frac{q_{2,n}(x)}{x^{2d+1}} & = \kappa_2 \left(\lim_{x\to \infty} \frac{(1-x^2) \rho^{\prime}(x)}{x^{d+1}}-\lim_{x\to \infty} \frac{A_{3,n}(x)}{x^{d+1}}\right) = \kappa_2  \left(-d+n+d\right)\\
& = \begin{cases}
n, & \text{if }\;  \dgr{B_{2,n-1}}<d-1,\\
\dsty                                                                                              n+\frac{n\Lambda_{n-1}}{\gamma_{2,n-1}\,h^{\alpha,\beta}_{n-2} }, & \text{if }\;  \dgr{B_{2,n-1}}=d-1,
\end{cases}
\end{align*}
where, according to \eqref{LambdaNpositive}, $ \dsty  \Lambda _{n-1} >0$,  i.e.,  $\dgr{q_{2,n}} = 2d+1$.

  \item
$$ \lim_{x\to \infty} \frac{q_{4,n}(x)}{x^{2d}}    = \left\{   \begin{array}{ll}
-\frac{\widehat{b}_{n-1}}{\gamma_{2,n-1}}, & \! \!  \text{if }   \dgr{B_{2,n-1}}<d-1,\\
  \dsty                                                                                                -\frac{\widehat{b}_{n-1}+(2n+\alpha+\beta-1)\frac{\Lambda_{n-1}}{h_{n-2}^{\alpha,\beta}}}{\gamma_{2,n-1}}, &\! \! \text{if } \dgr{B_{2,n-1}}=d-1.                                                                                               \end{array}                                                                                               \right.
$$
Then, according to \eqref{LambdaNpositive},    $\dgr{q_{4,n}} = 2d$.
  \item
\begin{align*}
\lim_{x\to \infty} \frac{q_{3,n}(x)}{x^{2d+1}}  & = -d \kappa_2 -\lim_{x\to \infty} \frac{D_{3,n}(x)}{x^{d+1}} \\
& = \begin{cases}
-(n+\alpha+\beta), & \text{if }\;  \dgr{B_{2,n-1}}<d-1,\\
\dsty
-(n+\alpha+\beta)\left(1+\frac{\Lambda_{n-1}}{\gamma_{2,n-1}\,h^{\alpha,\beta}_{n-2} }\right), & \text{if }\;  \dgr{B_{2,n-1}}=d-1,
\end{cases}
\end{align*}
where, according to \eqref{LambdaNpositive}, $ \dsty  \Lambda _{n-1} >0$,  i.e.,  $\dgr{q_{3,n}} = 2d+1$.
\end{enumerate}
\end{proof}

In the previous theorem, the polynomials $q_{k,n}$ were defined. Note that these polynomials are closely related to certain determinants. The following result summarizes some of their properties that will be of interest later. For brevity, we introduce the following~\mbox{notations}:
\begin{align*}
\Delta_{1,n}(x)&=B_{3,n}(x)A_{2,n}(x)-A_{3,n}(x)B_{2,n}(x). \\
\Delta_{2,n}(x) &=B_{3,n}(x)C_{2,n}(x)- A_{3,n}(x)D_{2,n}(x). \\
\Delta_{3,n}(x) & =B_{2,n}(x)C_{3,n}(x)-A_{2,n}(x)D_{3,n}(x).
\end{align*}

\begin{lemma}\label{LemmaDeterminant} Let $\dsty \rho_{N}(x) =\prod_{j=1}^{N} \left( x-c_j\right)$ and $\dsty \rho_{d-N}(x)=\prod_{j=1}^{N} \left( x-c_j\right)^{d_j}=\frac{\rho(x)}{\rho_{N}(x)}$. Then, the above polynomial determinants admit the following decompositions:

\begin{equation}\label{EqusLemma5}
\begin{aligned}
\Delta_{1,n}(x)&= \rho_{d-N}(x) \; \varphi_{1,n}(x), \quad \text{where }\;\dgr{\varphi_{1,n}}=d+N.\\
 \Delta_{2,n}(x)&=  \rho_{d-N}(x)\; \varphi_{2,n}(x), \quad \text{where }\;\dgr{\varphi_{2,n}}=d+N+1 .\\
  \Delta_{3,n}(x)&=  \rho_{d-N}(x) \; \varphi_{3,n}(x), \quad \text{where }\;\dgr{\varphi_{3,n}}=d+N+1 .
\end{aligned}
\end{equation}
\end{lemma}

\begin{proof} Multiplying \eqref{LemConnForm1}  by $B_{3,n}$ and \eqref{LemConnForm1-Dx} by $B_{2,n}$ and taking their difference, we have
\begin{align*}
\Delta_{1,n}(x) P^{\alpha,\beta}_n(x)  = &\; \rho(x) B_{3,n}(x) S_n(x)-(1-x^2) B_{2,n}(x)  \left(\rho^{\prime}(x) S_n(x)+\rho(x)   S_n^{\prime}(x) \right)\\
= & \; \rho_{d-N}(x) \Big( \rho_{N}(x)  B_{3,n}(x) S_n(x)-(1-x^2)B_{2,n}(x) \\
 & \Big(  \sum_{j=1}^{N}(d_j+1)\, \rho_{j,d_j}(x) \,S_n(x) +\rho_{N}(x)  \, S_n^{\prime}(x)\Big)\Big).
\end{align*}

As $P^{\alpha,\beta}_n(c_j) \neq 0$ for $j=1,\dots,N$ and $\dgr{\Delta_{1,n}}=\dgr{q_{1,n}}=2d$ (see the proof of Theorem~\ref{lemma44}), then there exists a polynomial  $\varphi_{1,n}$ of the degree $d+N$ such that $\Delta_{1,n}(x) = \rho_{d-N}(x) \; \varphi_{1,n}(x)$.

For the decomposition of $\Delta_{2,n}$ ($\Delta_{3,n}$) the procedure of the proof is analogous, using the linear system of \eqref{LemConnForm1-Dx} and \eqref{LemConnForm2} (\eqref{LemConnForm1}--\eqref{LemConnForm2-Dx}).
\end{proof}

\section{Ladder Jacobi-Sobolev Differential Operators and Consequences}\label{Sec-DiffEqn}

\begin{definition}[Ladder Jacobi-Sobolev differential operators] Let $\mathfrak{I}$ be  the identity operator. We define the two ladder differential operator on $\PP$ as
\begin{align*}
\mathfrak{L}_{n}^{\downarrow} & := A_{4,n}(x) \mathfrak{I}+B_{4,n}(x) \frac{d}{dx}\quad  \text{(lowering Jacobi-Sobolev differential operator)},\\
 \mathfrak{L}_{n}^{\uparrow} &:= C_{4,n}(x) \mathfrak{I}+D_{4,n}(x) \frac{d}{dx} \quad \text{(raising Jacobi-Sobolev differential operator)} .
\end{align*}
\end{definition}

\begin{remark} \label{Remark-Operators}Assume in \eqref{GeneralSIP} that  $d\mu(x)=d\mu^{\alpha,\beta}(x)=(1-x)^{\alpha}(1+x)^{\beta} dx$   $\left(\alpha,\beta>-1\right)$,  whose support is $[-1,1]$ and $\lambda_{j,k}\equiv 0$ for all pairs $(j,k)$. Under these conditions, it is not difficult to verify that $\mathfrak{L}_{n}^{\downarrow} \equiv \widehat{\mathfrak{L}}_{n}^{\downarrow} $ and $\mathfrak{L}_{n}^{\uparrow} \equiv \widehat{\mathfrak{L}}_{n}^{\uparrow} $.
\end{remark}

Now,  we can rewrite the ladder Equations \eqref{laddereq1} and \eqref{laddereq2} as%
\begin{align}
\mathfrak{L}_{n}^{\downarrow} \left[S_{n}(x) \right]&=\left( A_{4,n}(x)\mathfrak{I}+B_{4,n}(x)\frac{d}{dx}\right)S_{n}(x) =S_{n-1}(x),  \label{laddereq1rew} \\
\mathfrak{L}_{n}^{\uparrow}\left[ S_{n-1}(x)\right]&=\left( C_{4,n}(x)\mathfrak{I}+D_{4,n}(x)\frac{d}{dx}\right)S_{n-1}(x)  = S_{n}(x).  \label{laddereq2rew}
\end{align}%

In this section, we state several consequences of Equations \eqref{laddereq1rew} and \eqref{laddereq2rew}, which generalize known results for classical Jacobi polynomials to the Jacobi-Sobolev case.

First, we are going to obtain a second-order differential  equation with polynomial coefficients for $S_{n}$. The procedure  is well known and consists in applying the
raising operator $\mathfrak{L}_{n}^{\uparrow}$\ to both sides of the formula $ \mathfrak{L}_{n}^{\downarrow}\left[S_{n}\right]=S_{n-1}$. Thus, we have%
\begin{align*}
  0 = &  \mathfrak{L}_{n}^{\uparrow} \left[\mathfrak{L}_{n}^{\downarrow}\left[S_{n}(x)\right]\right]-S_{n}(x) \\
  = &  B_{4,n}(x)D_{4,n}(x)S_{n}^{\prime \prime }(x) \\
& + \left(A_{4,n}(x)D_{4,n}(x)+B_{4,n}(x)C_{4,n}(x)+D_{4,n}(x)B_{4,n}^\prime(x)\right)S_{n}^{\prime}(x)    \\
&+ \left(A_{4,n}(x)C_{4,n}(x)+D_{4,n}(x)A_{4,n}^\prime(x)-1\right) S_{n}(x) \\
=&\frac{q_{0,n}^2(x)}{q_{1,n}(x)q_{4,n}(x)} \; S_{n}^{\prime \prime }(x)\\
&  + \frac{q_{0,n}(x) \left(q_{1,n}(x)q_{2,n}(x)+q_{1,n}(x)q_{3,n}(x)+q_{0,n}^{\prime}(x) q_{1,n}(x)-q_{0,n}(x) q_{1,n}^{\prime}(x)\right)}{q_{4,n}(x) q_{1,n}^2(x)} \;S_{n}^{\prime}(x) \\
& + \left(\frac{q_{1,n}(x)q_{2,n}(x) q_{3,n}(x)+q_{0,n}(x) \left(q_{2,n}^{\prime}(x) q_{1,n}(x)-q_{2,n}(x) q_{1,n}^{\prime}(x)\right)}{q_{4,n}(x) q_{1,n}^2(x)}-1\right) S_{n}(x),
\end{align*}

from where we conclude the following result.
\begin{theorem}\label{Poly-HolEq} The $n$th monic orthogonal polynomial with
respect to the inner product \eqref{GeneralSIP} is a polynomial
solution of the second-order linear differential equation, with polynomial  coefficients%
\begin{equation}\label{DifEq-PolyCoef}
\mathfrak{P}_{2,n}(x)S_{n}^{\prime \prime }(x)+\mathfrak{P}_{1,n}(x)S_{n}^{\prime }(x)+\mathfrak{P}_{0,n}(x) S_{n} (x)=0,
\end{equation}%
where%
\begin{equation}\label{DiffEqn-CoefSobolev}
\begin{aligned}
 \mathfrak{P}_{2,n}(x) = &  q_{1,n}(x) q_{0,n}^2(x), \\
 \mathfrak{P}_{1,n}(x)= &   q_{0,n}(x) \left(q_{1,n}(x)q_{2,n}(x)+q_{1,n}(x)q_{3,n}(x)+q_{0,n}^{\prime}(x) q_{1,n}(x)-q_{0,n}(x) q_{1,n}^{\prime}(x)\right),\\
\mathfrak{P}_{0,n}(x)= & q_{1,n}(x)q_{2,n}(x) q_{3,n}(x)+q_{0,n}(x) \left(q_{2,n}^{\prime}(x) q_{1,n}(x)-q_{2,n}(x) q_{1,n}^{\prime}(x)\right)\\ &  - q_{4,n}(x) q_{1,n}^2(x),\\
&\dgr{\mathfrak{P}_{2,n}}=6d+4, \; \dgr{\mathfrak{P}_{1,n}}\leqslant 6d+3 \; \text{, and }\; \dgr{\mathfrak{P}_{0,n}}\leqslant 6d+2.
\end{aligned}
\end{equation}
\end{theorem}

\begin{remark}[The classical Jacobi differential equation] Under the conditions stated  in Remark~\ref{Remark-Operators},   \eqref{GeneralSIP}   becomes to the classical Jacobi inner product and $S_{n}(x) = P_{n}^{\alpha,\beta}$(x).

Note that, here, $A_{1,n}(x)\equiv 1$, $B_{1,n}(x)=0$ and $\rho(x) \equiv 1$. For the rest of the expressions involved in the coefficients of the differential Equation \eqref{DifEq-PolyCoef},  we have
\begin{equation*}
\begin{aligned}
\rho(x)   & \equiv 1, \; A_{1,n}(x)  \equiv  A_{2,n}(x) \equiv D_{2,n} (x)= 1,\; B_{1,n}(x)  \equiv  B_{2,n}(x) \equiv C_{2,n} (x)\equiv 0, \\
\Delta_n(x) & \equiv 1, \; A_{3,n}(x) = \widehat{a}_n(x), B_{3,n}(x)    = \widehat{b}_n,\; C_{3,n}(x)  = - \gamma_{2,n-1}^{-1} \widehat{b}_{n-1} \;\text{ and} \\  D_{3,n}(x)   & = \widehat{a}_{n-1}(x)+\gamma_{2,n-1}^{-1} \widehat{b}_{n-1}(x-\gamma_{1,n-1}).
\end{aligned}
\end{equation*}
Thus,
\begin{equation}\label{DiffEqn-CoefJacob}
\begin{aligned}
q_{0,n}(x)&= \left(1-x^2\right),\; q_{1,n}(x)=\widehat{b}_n,\quad  q_{2,n}(x) =- \widehat{a}_n(x),\\
 q_{3,n}(x)  & = -\widehat{a}_{n-1}(x)-\gamma_{2,n-1}^{-1} \widehat{b}_{n-1} (x-\gamma_{1,n-1})  \\ & = -\left( n+\alpha+\beta\right) x + \frac{\left( n+\alpha+\beta\right) \, \left(\alpha-\beta\right) }{2 n+\beta+\alpha}\;  \text{ and } \\  q_{4,n}(x) & = - \gamma_{2,n-1}^{-1} \widehat{b}_{n-1} = -\left( 2 n+\alpha+\beta-1\right).
\end{aligned}
\end{equation}

Substituting \eqref{DiffEqn-CoefJacob} in \eqref{DiffEqn-CoefSobolev}, the reader can verify that the differential Equation \eqref{DifEq-PolyCoef} becomes \eqref{Jacobi-DiffEqn},  i.e.,
\begin{align*}
\mathfrak{P}_{2,n}(x) &=  \left(1-x^2\right) , \;   \mathfrak{P}_{1,n}(x) =   \beta-\alpha-(\alpha+\beta+2) x \; \text{ and } \;  \mathfrak{P}_{0,n}(x) =  n(n+\alpha+\beta+1).
\end{align*}
\end{remark}

Second, we can obtain the polynomial $n$th degree of the sequence  $\left\{S_{n}\right\}_{n\geqslant  0}$ as the repeated action  ($n$ times) of the raising differential operator on the first Sobolev-type  polynomial of the sequence (i.e., the polynomial of degree zero).

\begin{theorem}
\label{mainresult2}The $n$th  Jacobi-Sobolev polynomial  $S_{n}$ $(n\geqslant  0)$ can be given by%
\begin{equation*}
S_{n}(x)=\left( \mathfrak{L}_{n}^{\uparrow}\mathfrak{L}_{n-1}^{\uparrow}\mathfrak{L}_{n-2}^{\uparrow}\cdots \mathfrak{L}_{1}^{\uparrow}\right) S_{0}(x),
\end{equation*}%
where $S_{0}(x)=1$.
\end{theorem}


\begin{proof} Using \eqref{laddereq2rew}, the theorem follows  for $n=1$. Next, the expression for $S_{n}$ is a straightforward consequence of the definition
of the raising operator. \end{proof}

To conclude this section, we prove an interesting three-term recurrence relation with rational coefficients, which satisfies the Jacobi-Sobolev monic polynomials. From the
explicit expression of the ladder operators, shifting $n$ to $n+1$ in \eqref{laddereq2rew}, we obtain
\begin{align*}
C_{4,n}(x)S_{n}(x)+D_{4,n}(x) \frac{d}{dx}S_{n}(x) &=S_{n-1}(x), \\
A_{4,n}(x) S_{n}(x)+B_{4,n}(x)\frac{d}{dx}S_{n}(x) &= S_{n+1}(x).
\end{align*}
Next, we multiply the first equation by $-B_{4,n}(x)$\ and the second equation by $D_{4,n}(x)$, and adding two resulting equations, we have the following three-term recurrence reaction with rational coefficients for the Jacobi-Sobolev monic orthogonal polynomials.

\begin{theorem} Under the assumptions of Theorem \ref{Poly-HolEq}, we have the recurrence relation

\begin{equation}\label{3TRR-Sob}
\begin{aligned}
 {q_{4,n+1}(x)q_{0,n}(x)} S_{n+1}(x) =&  \left[ q_{3,n+1}(x)q_{0,n}(x) - q_{2,n}(x)q_{0,n+1}(x) \right]S_{n}(x)\\ & + {q_{1,n}(x)q_{0,n+1}(x)} S_{n-1}(x),
\end{aligned}%
\end{equation}
where the explicit formula of the coefficient is given in Theorem  \ref{lemma44}.
\end{theorem}

\begin{proof}
From  \eqref{laddereq1}, and   \eqref{laddereq2} for $n+1$,  we have
\begin{align*}
q_{2,n}(x)S_{n}(x)+q_{0,n}(x)(x)S_{n}^{\prime}(x) & = q_{1,n}(x)S_{n-1}(x).\\
q_{3,n+1}(x)S_{n}(x)+q_{0,n+1}(x)S_{n}^{\prime}(x) & = q_{4,n+1}(x)S_{n}(x).
\end{align*}

Multiplying by $q_{0,n+1}(x)$ and $q_{0,n}(x)$, respectively, we subtract both equations to eliminate the derivative term obtaining
\begin{align*}
& \left(q_{3,n+1}(x)q_{0,n}(x)-q_{2,n}(x)q_{0,n+1}(x)\right)S_n(x) \\
& = q_{4,n+1}(x)q_{0,n}(x)S_{n+1}(x)-q_{1,n+1}(x)q_{0,n+1}(x)S_{n-1}(x),
\end{align*}
which is the required formula.
\end{proof}

\begin{remark}[The classical Jacobi three-term recurrence relation] Under the assumptions of  Remark \ref{Remark-Operators},   substituting \eqref{DiffEqn-CoefJacob} in \eqref{3TRR-Sob}, the reader can verify that the three-term recurrence relation~\eqref{3TRR-Sob} becomes \eqref{DifEq-PolyCoef},  i.e.,
\begin{align*}
\frac{q_{3,n+1}(x)q_{0,n}(x) - q_{2,n}(x)q_{0,n+1}(x)}{q_{4,n+1}(x)q_{0,n}(x)} & = x-\gamma_{1,n} \quad \text{and} \quad \frac{q_{1,n}(x)q_{0,n+1}(x)}{q_{4,n+1}(x)q_{0,n}(x)}  = -\gamma_{2,n}.
\end{align*}
\end{remark}

\section{Electrostatic Interpretation}\label{SecElectro}

Let us begin by recalling the  definition of a sequentially ordered  Sobolev inner product, which was stated in \cite{DiPiPe20} (Definition 1) or \cite{DiPiQui22} (Definition 1).

\begin{definition}\label{Set-SOrdered-General}
Let $\{(r_j,\nu_j)\}_{j=1}^M\!\subset \!\RR\!\times\!\ZZp$ be a finite sequence of $M$ ordered pairs and $A\subset \RR$. We say that $\{(r_j,\nu_j)\}_{j=1}^M $ is  sequentially ordered with respect to $A$, if
\begin{enumerate}
\item $0\leqslant \nu_1\leqslant \nu_2\leqslant \cdots \leqslant \nu_M$.
\item $r_k\notin\inter{\ch{A \cup\{r_1,r_2,\dots,r_{k-1}\}}}\dsty$ for $k=1,2,\dots,M$, where $\inter{\ch{B}}$ denotes the interior of the convex hull of an arbitrary set $B\subset \CC$.
\end{enumerate}

If $A=\emptyset$, we say that  $\{(r_j,\nu_j)\}_{j=1}^M $ is  sequentially ordered   for brevity.

We say that the discrete Sobolev inner product \eqref{GeneralSIP} is  sequentially ordered  if the set of ordered pairs $\{(c_j,i): 1\leqslant j\leqslant N, 0\leqslant i\leqslant d_j \text{ and } \eta_{j,i}>0 \}$ may be arranged to form a finite sequence of ordered pairs, which is sequentially ordered with respect to $(-1,1)$.
\end{definition}

 From the second condition of Definition \ref{Set-SOrdered-General}, the coefficient $\lambda_{j,d_j}$ is the only coefficient $\lambda_{j,i}$ ($i=0,1,\dots,d_j$) different from zero,  for each $j=1,2,\dots,N$. Hence,  \eqref{GeneralSIP} takes the form
\begin{align}\label{IP-Sobolev-SO}
\IpS{f}{g}=  \int_{-1}^{1} f(x) g(x) \,d\mu^{\alpha,\beta}(x) +\sum_{j=1}^{N} \lambda_{j,d_j} \,f^{(d_j)}(c_{j}) g^{(d_j)}(c_{j}),
\end{align}
where $d\mu^{\alpha,\beta}(x)=(1-x)^{\alpha}(1+x)^{\beta} dx$, with $\alpha,\beta>-1$.

Hereinafter, we will restrict our attention to sequentially ordered discrete Sobolev inner products. The following two lemmas  show  our reasons for this restriction.

\begin{lemma}[{\cite[Th. 1]{DiPiPe20}  and \cite[Prop. 4]{DiPiQui22}}]
\label{Th_ZerosSimp}  If \eqref{IP-Sobolev-SO}   is  a sequentially ordered discrete Sobolev inner product, then  $S_n$ has at least $n-N$ changes of sign on $(-1,1)$.
\end{lemma}

\begin{lemma}[{\cite[Lem. 3.4]{DiPiPe20}   and \cite[Th. 7]{DiPiQui22}}] \label{AsymCBounded_Lemma} Let   \eqref{IP-Sobolev-SO} be a sequentially ordered Sobolev inner product. Then, for all $n$ sufficiently large,   each sufficiently small neighborhood of  $c_j$, $j=1,\dots,N$, contains exactly one zero of $S_n$, and the remaining $n-N$ zeros lie on $(-1,1)$.
   \end{lemma}

As the coefficient of  $S_n$ is real, under the same hypotheses of  Lemma \ref{AsymCBounded_Lemma},   for all $n$ sufficiently large,  the zeros of $S_n$ are real and simple.

In the rest of this section, we will assume that the zeros of $S_n$ are simple. Note that  sequentially ordered  Sobolev inner products provide us with a wide class of Sobolev inner products  such that the zeros of the corresponding orthogonal polynomials are simple.  Therefore, for all $n$ sufficiently large, we have
\begin{equation*}
\begin{array}{lll}
 S_{n}^{\prime}(x)={\displaystyle\sum\limits_{i=1}^{n}} \prod\limits_{\substack{ j=1,  \\ j\neq i}}^{n}{(x-x_{n,j})}, &  &
 S_{n}^{\prime\prime}(x)={\displaystyle\sum\limits_{i=1}^{n}\sum\limits_{\substack{ j=1,  \\ j\neq i}}^{n}\prod\limits_{\substack{ l=1,  \\ i\neq j\neq l}}^{n}(x-x_{n,l})}, \\
&  &  \\
S_{n}^{\prime}(x_{n,k})={\displaystyle\prod\limits_{\substack{ j=1, \\ j\neq k}}^{n}(x_{n,k}-x_{n,j})}, &  & S_{n}^{\prime\prime}(x_{n,k})={\displaystyle2\sum\limits_{\substack{ i=1, \\ i\neq k}}^{n}\prod\limits_{\substack{ j=1,  \\ i\neq j\neq k}}^{n}(x_{n,k}-x_{n,j})}.
\end{array}
\end{equation*}

Now we evaluate the polynomials $\mathfrak{P}_{2,n}(x),\; \mathfrak{P}_{1,n}(x) $, and $\mathfrak{P}_{0,n}(x) $ in \eqref{DifEq-PolyCoef} at $x_{n,k}$, where $\left\{ x_{n,k}\right\}_{k=1}^{n}$ are the zeros of $S_{n}(x)$ arranged in an increasing order. Then, for $k=1,2,\dots,n$, we obtain \vspace{6pt}
\begin{align}\nonumber
0=& \mathfrak{P}_{2,n}(x_{n,k}) S_{n}^{\prime\prime}(x_{n,k})+ \mathfrak{P}_{1,n}(x_{n,k})  S_{n}^{\prime}(x_{n,k})+\mathfrak{P}_{0,n}(x_{n,k})  S_{n}(x_{n,k})\\
\nonumber =&  \mathfrak{P}_{2,n}(x_{n,k}) S_{n}^{\prime\prime}(x_{n,k})+ \mathfrak{P}_{1,n}(x_{n,k})  S_{n}^{\prime}(x_{n,k}).\\ \label{CritPoint}
0= &\frac{ S_{n}^{\prime\prime}(x_{n,k})}{ S_{n}^{\prime}(x_{n,k})} +\frac{\mathfrak{P}_{1,n}(x_{n,k})}{\mathfrak{P}_{2,n}(x_{n,k}) }=2  \sum_{{{i=1} \atop {i \neq k}}}^{n} {\frac{1}{x_{n,k}-x_{n,i}}} +\frac{\mathfrak{P}_{1,n}(x_{n,k})}{\mathfrak{P}_{2,n}(x_{n,k}) }.
\end{align}

Let us recall that, from \eqref{EqusLemma5},
\begin{align*}
\varphi_{1,n}(x) & =  \frac{\Delta_{1,n}(x)}{\rho_{d-N}(x)}, \quad \dgr{\varphi_{1,n}}=d+N,\\
\varphi_{2,n}(x)& =  \frac{\Delta_{2,n}(x)}{\rho_{d-N}(x)}, \quad \dgr{\varphi_{2,n}}=d+N+1 ,\\
\varphi_{3,n}(x) & =   \frac{\Delta_{3,n}(x)}{\rho_{d-N}(x)}, \quad \dgr{\varphi_{3,n}}=d+N+1 .
\end{align*}

Hence, from Theorems \ref{lemma44} and \ref{Poly-HolEq}  and Lemma \ref{LemmaDeterminant},
\begin{align}\nonumber
  \frac{ \mathfrak{P}_{1,n}(x) }{ \mathfrak{P}_{2,n}(x) }=& \frac{q_{1,n}(x)q_{2,n}(x)+q_{1,n}(x)q_{3,n}(x)+q_{0,n}^{\prime}(x) q_{1,n}(x)-q_{0,n}(x) q_{1,n}^{\prime}(x)}{q_{1,n}(x) q_{0,n}(x)}\\ \nonumber
  =&\frac{q_{2,n}(x)+q_{3,n}(x)}{q_{0,n}(x)}+  \frac{ q_{0,n}^{\prime}(x)}{q_{0,n}(x)}- \frac{ q_{1,n}^{\prime}(x)}{q_{1,n}(x)}\\ \nonumber
  =& 2\frac{\rho^{\prime}(x)}{\rho(x)}+\frac{\Delta_{2,n}(x)+\Delta_{3,n}(x)}{ \left(1-x^2\right)\rho(x)\delta_n(x)}+\frac{ \Delta_{n}^{\prime}(x)}{\Delta_{n}(x)}+\frac{2x}{x^2-1}- \frac{ \Delta_{1,n}^{\prime}(x)}{\Delta_{1,n}(x)}\\ \nonumber
    =& 3\frac{\rho^{\prime}(x)}{\rho(x)}+\frac{\varphi_{2,n}(x)+\varphi_{3,n}(x)}{ \left(1-x^2\right)\rho_N(x)\delta_n(x)}+\frac{ \delta_{n}^{\prime}(x)}{\delta_{n}(x)}+\frac{1}{x-1} +\frac{1}{x+1}\\ &- \frac{ \varphi_{1,n}^{\prime}(x)}{\varphi_{1,n}(x)}-\frac{ \rho_{d-N}^{\prime}(x)}{\rho_{d-N}(x)}. \label{DescompCociente}
  \end{align}

Let us write  $\dsty  \frac{\rho^{\prime}(x)}{\rho(x)}  = \sum_{j=1}^{N} \frac{d_j+1}{x-c_j}.\qquad \frac{\rho_{d-N}^{\prime}(x)}{\rho_{d-N}(x)}  = \sum_{j=1}^{N} \frac{d_j}{x-c_j}.$

 As $\dsty \psi_1(x)=\varphi_{2,n}(x)+\varphi_{3,n}(x)$ and  $\dsty \psi_2(x)=\left(1-x^2\right)\rho_N(x)\delta_n(x)$ are polynomials of the degree $d+N+1$ and $d+N+2$, respectively, we have that $\dsty  \frac{\psi_1(x)}{ \psi_2(x)}$  is a rational proper fraction. Therefore,
$$
\frac{\psi_1(x)}{ \psi_2(x)} =- \frac{r(1)}{x-1}+\frac{r(-1)}{x+1}+\sum_{j=1}^{N} \frac{r(c_j)}{x-c_j}+\sum_{j=1}^{d} \frac{r(u_j)}{x-u_j}, \quad \text{where } r(x)=\frac{\psi_1(x)}{ \psi_2^{\prime}(x)}.
$$

Based on the results of our numerical experiments, in the remainder of the section, we will assume certain restrictions with respect to some  functions and parameters involved in~\eqref{DescompCociente}. In that sense, we suppose that
\begin{enumerate}
  \item The zeros of $\delta_{n}$ are real, simple,  and different from $x_{n,k}$ for all $k=1,\dots,n$. Therefore, $\dsty \delta_{n}(x)=  \prod_{k=1}^{d} (x-u_j)$, where $u_i\neq u_j$ if $i \neq j$,  and   $ \dsty \frac{\delta_{n}^{\prime}(x)}{\delta_{n}(x)}   = \sum_{j=1}^{d} \frac{1}{x-u_j}.$
  \item Let $\dsty \varphi_{1,n}(x)=  \kappa_{1}\prod_{j=1}^{N_1} (x-e_{j})^{\ell_{5,j}}$, where  $e_{j} \in \CC \setminus {\ch{[-1,1]\cup\{c_1,\dots,c_{N}\}}}$ for all $j=1,\dots, N-1$, and  $\dsty \sum_{j=1}^{N_1}\ell_{5,j}=d+N$. Therefore, $\dsty  \frac{\varphi_{1,n}^{\prime}(x)}{\varphi_{1,n}(x)}   = \sum_{j=1}^{N_1} \frac{\ell_{5,j}}{x-e_{j}}.$

 \item Substituting into \eqref{DescompCociente}  the previous decompositions, we have
$$
  \frac{ \mathfrak{P}_{1,n}(x) }{ \mathfrak{P}_{2,n}(x) }= \frac{\ell_{1}}{x-1}+\frac{\ell_{2}}{x+1}+\sum_{j=1}^{N} \frac{\ell_{3,j}}{x-c_j}+\sum_{j=1}^{d} \frac{\ell_{4,j}}{x-u_j}-  \sum_{j=1}^{N_1} \frac{\ell_{5,j}}{x-e_{j}},
$$where  $\ell_{1}=1-r(1)$, $\ell_{2}=1+r(-1) $, $\ell_{3,j}=2d_j+r(c_j)+3$, and   $\ell_{4,j}= r(u_j)+1$. We will assume that  $\ell_{1},\ell_{2},\ell_{3,j},\ell_{4,j}\geqslant 0$.
\end{enumerate}

From  \eqref{CritPoint}, for $k=1,\dots, n$,
\begin{align}\nonumber
 0=&\;   \sum_{{{i=1} \atop {i \neq k}}}^{n} {\frac{1}{x_{n,k}-x_{n,i}}} +\frac{\ell_{1}}{2}\,\frac{1}{x_{n,k}-1}+\frac{\ell_{2}}{2} \, \frac{1}{x_{n,k}+1} \\ & \,+\frac{1}{2}\sum_{j=1}^{N} \frac{\ell_{3,j}}{x_{n,k}-c_j}+\frac{1}{2}\sum_{j=1}^{d} \frac{\ell_{4,j}}{x_{n,k}-u_j}+\frac{1}{2}\sum_{j=1}^{N_1} \frac{\ell_{5,j}}{e_{j}-x_{n,k}}.  \label{PtosCrit-Potencial}
\end{align}

Let $\dsty \overline{\omega}=  (\omega_1,\omega_2,\cdots,\omega_n)$, $\overline{x}_n=  (x_{n,1},x_{n,2},\cdots,x_{n,n})$ and denote
\begin{align}  \label{LogPotential}
E(\overline{\omega}) := & \;\sum_{1\leq k<j \leq n }\log{\frac{1}{|\omega_j-\omega_k|}} +F(\overline{\omega})+ G(\overline{\omega}), \;  \\  \nonumber
F(\overline{\omega}):=&\frac{1}{2}\sum_{k=1}^{n}\left(\log{\frac{1}{|1-\omega_k|^{\ell_1}}} +\log{\frac{1}{|1+\omega_k|^{\ell_2}}} +\sum_{j=1}^{N}  \log{\frac{1}{|c_j-\omega_k|^{\ell_{3,j}}}}\right), \\
G(\overline{\omega}):=&\frac{1}{2}\sum_{k=1}^{n}\left(   \sum_{j=1}^{d}   \log{\frac{1}{|u_j-\omega_k|^{\ell_{4,j}}}}+  \sum_{j=1}^{N_1}  \log{\frac{1}{|e_j-\omega_k|^{\ell_{5,j}}}}\right).\nonumber
\end{align}

Let us introduce the following electrostatic interpretation:

\begin{quotation}
\emph{Consider the system of $n$ movable positive unit charges at $n$ distinct points of the real line, $\{ \omega_1,\omega_2,\cdots,\omega_n\}$, where their interaction obeys the logarithmic potential law (that is, the force is inversely proportional to the relative distance)  in the presence of the total external potential $V_n(\overline{\omega})=F(\overline{\omega})+ G(\overline{\omega})$. Then, $E(\overline{\omega})$ is the total energy  of this system. }
\end{quotation}

Following the notations   introduced in \cite{Ism00-A} (Section 2), the Jacobi-Sobolev inner product creates two external fields. One is a
long-range field whose potential is $F(\overline{\omega})$, and the other    is a short-range field whose potential is  $G(\overline{\omega})$. Therefore, the total external potential $V_n(\overline{\omega})$  is the sum of the short- and long-range potentials, which is dependent on $n$ (i.e., varying external potential).

Therefore, for each $k=1,\dots, n$, we have $\dsty \frac{\partial  E}{\partial \omega_k}(\overline{x}_n)=   0$;  i.e.,  the zeros of $S_n$ are the zeros  of  the gradient  of the total potential of energy $E(\overline{\omega})$ ($\nabla E(\overline{x}_n)=0$).

\begin{theorem}\label{PositiveHessian}
The zeros of $S_n(x)$ are a local minimum of $E(\overline{\omega})$,  if  for all $k=1,\dots, n;$
\begin{enumerate}
  \item $\dsty \frac{\partial  E}{\partial \omega_k}(\overline{x}_n)=   0$.
  \item $\dsty \frac{\partial^2  V_n}{\partial w_k^2}(\overline{x}_n) =\frac{\partial^2  F}{\partial w_k^2}(\overline{x}_n)+\frac{\partial^2 G}{\partial w_k^2}(\overline{x}_n) > 0. $
\end{enumerate}

\end{theorem}
\begin{proof}
The Hessian matrix of $E$ at $\overline{x}_n$ is given by
\begin{equation}\label{Hessian_M}
\nabla_{\overline{\omega}\,\overline{\omega}}^2 E(\overline{x}_n) = \begin{cases}
\dsty \frac{\partial^2 E}{\partial w_k\partial w_j}(\overline{x}_n)=-(x_k-x_j)^{-2}, & \text{if } \;k\neq j,\\
\dsty \frac{\partial^2 E}{\partial w_k^2}(\overline{x}_n)= \sum_{{{i=1} \atop {i \neq k}}}^{n} \frac{1}{(x_{n,k}-x_{n,i})^2}+\frac{\partial^2\left(V_n\right)}{\partial w_k^2}(\overline{x}_n), & \text{if } \;k = j.
\end{cases}
\end{equation}

 Note that  \eqref{Hessian_M} is a symmetric real matrix with negative values in the nondiagonal entries. Additionally, note that
$$
\sum_{{{j=1} \atop {i \neq k}}}^{n}\frac{\partial^2 E}{\partial w_k\partial w_j}(\overline{x}_n)+\dsty \frac{\partial^2 E}{\partial w_k^2}(\overline{x}_n) = \frac{\partial^2 V_n}{\partial w_k^2}(\overline{x}_n).
$$
Since this is positive, we conclude according to Gershgorin's theorem \cite{Horn90} (Theorem~6.1.1) that the eigenvalues of the Hessian are positive, and therefore, \eqref{Hessian_M} is positive definite. Combining this with the fact that $\nabla E(\overline{x}_n)=0$, we conclude that $\overline{x}_n$ is a local minimum of~\eqref{LogPotential}.
\end{proof}

The computations of the following examples have been performed  using the symbolic computer algebra system  \emph{Maxima} \cite{OchMak20}.
In all cases, we fixed $n=12$ and considered sequentially ordered Sobolev inner products  (see Definition \ref{Set-SOrdered-General} and  \mbox{Lemmas \ref{Th_ZerosSimp} and \ref{AsymCBounded_Lemma}}). From~\eqref{PtosCrit-Potencial}, it is obvious that $\nabla E(\overline{x}_{12})=0$, where $\overline{x}_{12}=  (x_{12,1},x_{12,2},\cdots,x_{12,n})$ and \mbox{$S_{12}(x_{12,k})=0$} for $k=1,\;2,\dots,\,12$. Under the above condition, $\overline{x}_{12}$ is a local minimum (maximum) of $E$ if the corresponding Hessian matrix at $\overline{x}_{12}$ is positive (negative) definite; in any other case, $\overline{x}_{12}$ is said to be a saddle point. We recall that a square matrix is positive (negative) definite if all its eigenvalues are positive (negative).

\begin{example}[Case in which the conditions of   Theorem \ref{PositiveHessian} are satisfied]\

\begin{enumerate}
  \item Jacobi-Sobolev inner product $\dsty \IpS{f}{g}=\int_{-1}^{1} f(x)g(x)(1+x)^{100}dx+f^{\prime}(2)g^{\prime}(2)$.
  \item Zeros of $S_{12}(x)$.
   \begin{align*}
     \overline{x}_{12}=& \left(0.44845, \, 0.563364, \, 0.653317, \, 0.728094, \, 0.791318, \, 0.844674, \right.  \\
       &\left.   0.889402, \, 0.925746, \, 0.954364, \, 0.97639, \, 0.989824, \, 0.998408 \right).
   \end{align*}
  \item Total potential of energy  $\dsty E(\overline{\omega}) =\sum_{1\leq k<j \leq 12 }\log{\frac{1}{|\omega_j-\omega_k|}} + F(\overline{\omega})+G(\overline{\omega})$,   where
  \begin{align*}
F(\overline{\omega})= &\frac{1}{2}\sum_{k=1}^{12} \left( \log{\frac{1}{\left| \omega_k-1 \right|} }+ \log{\frac{1}{\left| \omega_k+1 \right|^{101}} }+ \log{\frac{1}{\left| \omega_k-2 \right|^3} }\right),\; \\
 G(\overline{\omega})= & \frac{1}{2}\sum_{k=1}^{12} \log\left| (\omega_k - 1.04563) \tau(\omega_k)\right| \;
\text{ and }  \; \tau(x)  =  x^2-3.8812x+3.76606 >  0.
\end{align*}
  \item From \eqref{PtosCrit-Potencial}, $\dsty \frac{\partial  E}{\partial \omega_j}(\overline{x}_{12})=0, $ for $j=1,\dots, 12$.
  \item Computing the corresponding Hessian matrix at $\overline{x}_{12}$, we have that the approximate values of its eigenvalues  are
\begin{align*}
&   \left\{ 81.7737, \, 220.5813, \, 383.5185, \, 586.5056, \, 857.6819, \, 1248.8, \; 1857.7, \;2927.5, \;5039.9,\right. \\
&   \left.  \; 9986.6, \; 26185, \; 214620 \right\}.
\end{align*}
\end{enumerate}

 Thus, Theorem \ref{PositiveHessian} holds for this example, and we have the required local electrostatic equilibrium distribution.
\end{example}


\begin{example}[Case in which the conditions of  Theorem \ref{PositiveHessian} are satisfied]\

\begin{enumerate}
  \item Jacobi-Sobolev inner product $$\dsty \IpS{f}{g}=\int_{-1}^{1} f(x)g(x)(1+x)^{110}dx+f^{\prime}(1)g^{\prime}(1) + f^{\prime\prime}(2)g^{\prime\prime}(2).$$
  \item Zeros of $S_{12}(x)$.
   \begin{align*}
     \overline{x}_{12}=& \left(   0.482433,\,0.590159, \; 0.674139,\, 0.74379,\, 0.802629,\, 0.852355, \right.  \\
       &\left.  0.894142,\, 0.928255,\, 0.955716,\, 0.976239, \, 0.990307,\, 0.998211 \right).
   \end{align*}
  \item Total potential of energy  $\dsty E(\overline{\omega}) =\sum_{1\leq k<j \leq 12 }\log{\frac{1}{|\omega_j-\omega_k|}} + F(\overline{\omega})+G(\overline{\omega})$,   where
  \begin{align*}
F(\overline{\omega})= &\frac{1}{2}\sum_{k=1}^{12} \left( \log{\frac{1}{\left| \omega_k-1 \right|^{3}} }+ \log{\frac{1}{\left| \omega_k+1 \right|^{111}} }+ \log{\frac{1}{\left| \omega_k-2 \right|^4} }\right),\; \\
 G(\overline{\omega})= & \frac{1}{2}\sum_{k=1}^{12} \log\left|{(\omega_k -1.22268)(\omega_k -1.94089) \tau(\omega_k) }\right|\\
& \text{and } \; \tau(x)  =   x^2-3.8196x+3.65881>0.
\end{align*}
  \item From \eqref{CritPoint}, $\dsty \frac{\partial  E}{\partial \omega_j}(\overline{x}_{12})=0, $ for $j=1,\dots, 12$.
  \item Computing the corresponding Hessian matrix at $\overline{x}_{12}$, we have that the approximate values of its eigenvalues  are
\begin{align*}
&   \left\{102.3077,\; 265.8911,\;459.368,\;702.7009,\;1030.2,\;1504.8,\;2247.1,\;3563.2,\;6146,\right. \\
&   \left. 12806,\; 38783,\; 488410 \right\}.
\end{align*}
   \end{enumerate}

Thus, Theorem \ref{PositiveHessian} holds for this example, and we have the required local electrostatic equilibrium distribution.

\end{example}


\begin{example}[Case in which the conditions of   Theorem \ref{PositiveHessian} are not satisfied]\label{Examp_NoS}\

\begin{enumerate}
  \item Jacobi-Sobolev inner product $\dsty \IpS{f}{g}=\int_{-1}^{1} f(x)g(x)dx+f^{\prime}(2)g^{\prime}(2)$.
  \item Zeros of $S_{12}(x)$.
   \begin{align*}
     \overline{x}_{12}=& \left(-0.979635,\, -0.894154,\, -0.746211,\, -0.545446,\,-0.305098,\, -0.0412552,\right.  \\
       &\left. 0.227973,\, 0.483321,\, 0.705221,\, 0.87481,\, 0.975632,\, 2.1607 \right).
   \end{align*}
  \item Total potential of energy  $\dsty E(\overline{\omega}) =\sum_{1\leq k<j \leq 12 }\log{\frac{1}{|\omega_j-\omega_k|}} + F(\overline{\omega})+G(\overline{\omega})$,   where
      \begin{align*}
F(\overline{\omega})= &\frac{1}{2}\sum_{k=1}^{12} \left( \log{\frac{1}{\left| \omega_k-1 \right|} }+ \log{\frac{1}{\left| \omega_k+1 \right|} }+ \log{\frac{1}{\left| \omega_k-2 \right|^3} }\right),\\
 G(\overline{\omega})= & \frac{1}{2}\sum_{k=1}^{12} \log\left|{(\omega_k -2.12065) \tau(\omega_k) }\right|\;
\text{and } \; \tau(x)  =   {x}^{2}-3.74216\,x+3.51112>  0.
\end{align*}
  \item From \eqref{PtosCrit-Potencial}, $\dsty \frac{\partial  E}{\partial \omega_j}(\overline{x}_{12})=0, $ for $j=1,\dots, 12$.
  \item Computing the corresponding Hessian matrix at $\overline{x}_{12}$, we have that the approximate values of its eigenvalues  are
\begin{align*}
&   \left\{1388.3 , \, 975.7989, \,242.5338 , \, 179.5748, \,107.6368 , \, 86.754, \;70.7275, \,62.6406, \right. \\
&   \left. \,50.3046,\, 34.4135, \,14.0599 , \,  -258.3366\right\}.
\end{align*}
\end{enumerate}

 Then,  $\overline{x}_{12}$ is a saddle point  of $E(\overline{\omega})$.
\end{example}

\begin{remark}
As can be noticed, in some cases, the configuration given by the external field includes complex points; they correspond to $e_j$. Specifically, in the examples, these points are given as the zeros of $\tau(x)$. Since $\phi_{1,n}(x)$ is a polynomial of real coefficients, the nonreal zeros arise as complex conjugate pairs. Note that
$$
\frac{a}{x-z}+\frac{a}{x-\overline{z}} = a\frac{2x+2\Re{z}}{x^2+2\Re{z}+|z|^2}
$$
where $\Re z$ denotes the real part of $z$. The antiderivative of the previous expression is $a\ln(x^2+2\Re z+ |z|^2)$. This means in our current case that the presence of complex roots does not change the formulation of the energy function.
\end{remark}


\subsection*{What Happens If the Hessian Is Not Positive Definite?  A Case Study}

Theorem \ref{PositiveHessian}  gives  us a general condition to determine whether the electrostatic interpretation is a mere extension of the classical cases. However, in  Example \ref{Examp_NoS}, the Hessian has one negative eigenvalue of about $-258$ corresponding to the last variable $\omega_n$. Therefore, we do not have the nice interpretation given in Theorem \ref{PositiveHessian}. However, note that the rest of the eigenvalues are positive, which means that the number
$$\frac{\partial^2 (V_n)}{\partial w_k^2}(\overline{x}_n)$$
remains positive for $k=1,\ldots, 11$. In this case,  the potential function exhibits a saddle point. The presence of the saddle point is somehow justified by the attractor 
point $a \approx -2.121$ having a zero ( $x_{12,12} \approx 2.161$) in its neighborhood. In this case, we are able to give an interpretation of the position of the zeros by considering a problem of conditional extremes.

 Assume that, when checking the Hessian, we obtained that the eigenvalues $\lambda_{i}$, for $i \in \mathcal{E}\subset \lbrace 1,2,\ldots,n \rbrace$, are negative or zero. Without loss of generality, assume that this happens for the last $m_{\mathcal{E}}=|\mathcal{E}|$ variables.  This is a saddle point. However, the   rest of the eigenvalues are positive,  which means that the truncated Hessian $\nabla_{\omega_{m_{\mathcal{E}}}\omega_{m_{\mathcal{E}}}}^2 E$ formed by taking the first $n-m_{\mathcal{E}}$ rows and columns of $\nabla_{\overline{\omega}\,\overline\omega}^2 E_R$ is a positive definite matrix by the same arguments used in the proof of Theorem \ref{PositiveHessian}.

Let us define the following problem of conditional extremum  on $\overline \omega = \overline \omega_n\in \RR^n$

$$
\begin{aligned}
& \min_{\overline{\omega}_n\in \RR^n} E(\overline \omega_n)\\
&\text{subject to } \omega_k - x_k = 0, \; \text{for all }\;  k=n-m_{\mathcal E}+1,\ldots, n.
\end{aligned}
$$

Note that this problem is equivalent to solve

\begin{equation*}
\min_{\overline \omega_{n-m_\mathcal{E}} \in \RR^{n-m_\mathcal{E}}}  E_R(\overline \omega_{n-m_\mathcal{E}},x_{m_\mathcal{E}+1},\ldots,x_n).
\end{equation*}

Let us prove that $\overline x_{n-m_\mathcal{E}}$ is a minimum of this problem. Note that the gradient of this function corresponds to the first $n-m_\mathcal{E}$ conditions of \eqref{PtosCrit-Potencial}, and the second-order condition is given by the truncated Hessian $\nabla_{\omega_{m_{\mathcal{E}}}\omega_{m_{\mathcal{E}}}}^2 E(\overline x_{m_{\mathcal{E}}})$, which is  by hypothesis positive definite.

Therefore, the configuration $\overline{x}_n$ corresponds to the local equilibrium of the energy function \eqref{LogPotential} once $m_{\mathcal{E}}$ charges are fixed.

\end{document}